\documentclass[final,leqno]{siamltex}

\usepackage{amsfonts,amssymb} 
\usepackage{amsmath} 
\usepackage{color} 
\usepackage{graphicx} 
\usepackage{epsfig,mathrsfs} 
\usepackage[bf,SL,BF]{subfigure} 
\usepackage{fancyhdr} 
\usepackage{CJK} 
\usepackage{caption} 
\usepackage{wrapfig} 
\usepackage{cases} 
\usepackage{bm}
\usepackage{algorithm}
\usepackage{algorithmic}
\newtheorem{remark}{Remark}[section] 
\newtheorem{example}{Example}[section] 

\usepackage{booktabs}
\usepackage{mathtools}

\title{A sharp error estimate of piecewise polynomial collocation  for nonlocal problems with weakly singular kernels
\thanks{This work was supported by NSFC 11601206 and the Fundamental Research Funds for the Central Universities under Grant No. lzujbky-2019-80. }}

\author{Minghua Chen\thanks{Corresponding author. School of Mathematics and Statistics, Gansu Key Laboratory of Applied Mathematics and Complex Systems,
 Lanzhou University, Lanzhou 730000, P.R. China  (Email: chenmh@lzu.edu.cn)}
\and Wenya Qi \thanks{
School of Mathematics and Statistics, Gansu Key Laboratory of Applied Mathematics and Complex Systems,
 Lanzhou University, Lanzhou 730000, P.R. China  (Email: qiwy16@lzu.edu.cn) }
 \and Jiankang Shi \thanks{
School of Mathematics and Statistics, Gansu Key Laboratory of Applied Mathematics and Complex Systems,
 Lanzhou University, Lanzhou 730000, P.R. China  (Email: shijk17@lzu.edu.cn)}
  \and Jiming Wu \thanks{
Institute of Applied Physics and Computational Mathematics, P.O. Box 8009,
 Beijing 100088, P.R. China  (Email: wu\underline{ }jiming@iapcm.ac.cn)}
 }

\begin{document}

\maketitle

\begin{abstract}
As is well known, using piecewise linear polynomial collocation (PLC) and piecewise quadratic  polynomial collocation (PQC), respectively,  to approximate
the weakly singular   integral
$$I(a,b,x) =\int^b_a \frac{u(y)}{|x-y|^\gamma}dy,  \quad x \in (a,b) ,\quad 0< \gamma <1,$$
have the local  truncation error $\mathcal{O}\left(h^2\right)$ and $\mathcal{O}\left(h^{4-\gamma}\right)$.
Moreover,  for Fredholm weakly singular integral equations of the second kind, i.e.,
$\lambda u(x)- I(a,b,x) =f(x)$ with $ \lambda \neq 0$,
also have    global convergence rate  $\mathcal{O}\left(h^2\right)$ and $\mathcal{O}\left(h^{4-\gamma}\right)$ in  [Atkinson and  Han,  Theoretical Numerical Analysis, Springer, 2009].

Formally, following  nonlocal  models   can be viewed  as  Fredholm weakly singular integral equations
$$\int^b_a \frac{u(x)-u(y)}{|x-y|^\gamma}dy  =f(x),   \quad x \in (a,b) ,\quad 0< \gamma <1.$$
However, there are still some significant  differences for the models in these two fields.
In the first part of this paper we prove that the weakly singular  integral by PQC have
an optimal  local truncation error  $\mathcal{O}\left(h^4\eta_i^{-\gamma}\right)$,
where $\eta_i=\min\left\{x_i-a,b-x_i\right\}$ and $x_i$ coincides with an element junction point.
Then a sharp   global convergence estimate  with $\mathcal{O}\left(h\right)$ and  $\mathcal{O}\left(h^3\right)$  by PLC and PQC,  respectively,   are established for nonlocal problems.
Finally, the numerical experiments including two-dimensional case are given to illustrate the effectiveness of the presented method.
\end{abstract}

\begin{keywords}
Nonlocal  problems, weakly singular kernels, piecewise polynomial collocation,  convergence analysis
\end{keywords}

\begin{AMS}
45F15, 65L60, 65M12
\end{AMS}

\pagestyle{myheadings}
\thispagestyle{plain}
\markboth{M. H. CHEN, W. Y. QI, J. K. SHI,  AND J. M. WU}{SHARP ERROR ESTIMATE  FOR NONLOCAL PROBLEMS }
\section{Introduction}

In this paper we study an error estimate of the   piecewise linear  polynomial  collocation (PLC) and  piecewise quadratic  polynomial  collocation (PQC)
for the  nonlocal problems  with a  weakly singular kernels,  whose prototype equation is   \cite{Andreu:10,Bates:06,Silling:00,TWW:13}
\begin{equation}\label{n1.1}
 \int^b_a \frac{u(x)-u(y)}{|x-y|^\gamma}dy  =f(x),   \quad x \in (a,b) ,\quad 0< \gamma <1
\end{equation}
with Dirichlet  boundary conditions  $u(a)=u_a$ and $u(b)=u_b$.
Such as nonlocal problems \eqref{n1.1} have been used to model very different scientific phenomena occurring in various applied fields, for example in materials science, biology, particle systems, image processing, coagulation models, mathematical finance, etc. \cite{Andreu:10,Bates:06}.

Formally, the nonlocal  models  \eqref{n1.1} can be viewed  as  Fredholm weakly singular integral equations of the second kind \cite{Aikinson:67,Aikinson:09,ZhangGJ:16}, i.e.,
\begin{equation*}\tag{$*$}
\lambda u(x)- \int^b_a \frac{u(y)}{|x-y|^\gamma}dy  =f(x),    \quad x \in (a,b) ,\quad 0< \gamma <1
\end{equation*}
with a nonzero complex number $\lambda \in \mathbb{C}$.
However, there are still some significant  differences for the models in these two fields.
For example, the inverse operators of  Fredholm  integral equations ($*$)  are uniformly bounded, see Theorem 12.5.1 of  \cite{Aikinson:09} or \cite{Aikinson:67}; but nonlocal model  \eqref{n1.1} is unbounded.
From perspective of error analysis,  it is shown that  the Fredholm integral equations  ($*$) have $\mathcal{O}\left(h^2\right)$ convergence  \cite[p.\,522]{Aikinson:09} by PLC
and $\mathcal{O}\left(h^{4-\gamma}\right)$ convergence  \cite[p.\,525]{Aikinson:09} by PQC.
Such a situation does not take place for  model \eqref{n1.1}, even for  the case $\gamma=0$. Later in the section 4,
 we  prove  an optimal global convergence estimate with $\mathcal{O}\left(h\right)$  by PLC and $\mathcal{O}\left(h^3\right)$ by PQC, respectively, for  model \eqref{n1.1}.
In fact, the convergence rate  for model \eqref{n1.1} with PLC remains to be proved in \cite{TWW:13}.

The first key step of error analysis for  models  \eqref{n1.1} is to study the  following   integral with the weakly singular kernels, being defined as
\begin{equation}\label{n1.2}
 I(a,b,x) =\int^b_a \frac{u(y)}{|x-y|^\gamma}dy,  \quad x \in (a,b) ,\quad 0< \gamma <1.
\end{equation}
It should be noted that the integral \eqref{n1.2}  can be  decomposed  into  Abel-Liouville integrals (often also called Riemann-Liouville fractional integrals) \cite{Lubich:86} and Weyl fractional integral \cite{Miller:93} if they depart  from the constant coefficient $1/\Gamma(1-\gamma)$.

Among various techniques for solving integral equations, collocation methods are among the simplest \cite{ZhangGJ:16}, which is only needed one-fold of integration  and is much simpler to implement on a computer.
Piecewise polynomial collocation methods for the integral \eqref{n1.2} have been extensively examined by many authors.
As is well known, for weakly singular ($0< \gamma< 1$)  integral \eqref{n1.2},   an optimal  error estimate with  $\mathcal{O}\left(h^2\right)$
was proved by PLC and only $\mathcal{O}\left(h^3\right)$ convergence was established by PQC in \cite{Aikinson:67}.
Up to now, the quasi-optimal  error estimate with  $\mathcal{O}\left(h^{4-\gamma}\right)$ convergence was provided by PQC, see \cite{Hoog:73} or \cite[p.\,525]{Aikinson:09}.
A few years later, the error estimate of the Newton-Cotes rules (piecewise polynomial collocation) for hypersingular ($\gamma\geq 1$) integrals \eqref{n1.2} was first studied in \cite{Linz:85}.
Later, the superconvergence estimate of the Hadamard finite-part (hypersingular) integral is discussed in \cite{WL:05,WS:08} and a class of collocation-type methods are developed in \cite{LiSun:10}.
Recently,  fractional hypersingular integral equations and nonlocal diffusion equations with PLC is studied  in \cite{ZhangGJ:16} and a general Newton-Cotes rules for fractional hypersingular integrals have been developed in \cite{GFTJZ:18}.
It should be noted that there are still some differences for the hypersingular integral and weakly singular integral equations. For example,  the stiffness matrix of hypersingular integral is a strictly diagonally dominant M-matrix \cite{ZhangGJ:16}, however, it is not possessed for the  weakly singular integral equations by PLC.

Numerical methods for the nonlocal problems \eqref{n1.1} have been proposed by various authors.
There are already the second-order convergence results for model \eqref{n1.1}  by linear  FEM  \cite{CES:19,WT:12}
and for  peridynamic or nonlocal problems with the horizon parameter by PLC \cite{CD:17,Tian:13,ZhangGJ:16}.
As with our previous reviews, it seems to be second-order convergence for nonlocal  model \eqref{n1.1} as well as Fredholm weakly singular integral equations ($*$) by PLC.
Unfortunately, the numerical result of \eqref{n1.1} with $\gamma=1$  shows that the convergence rate seems to be close to 1.5 by PLC \cite{TWW:13} although it remains to be proved.
In this work, inspired by these observations, we will provide  the rigorous convergence error estimate with $\mathcal{O}\left(h\right)$ by PLC for the nonlocal model (\ref{n1.1}),  even for  the case $\gamma=0$.
How about PQC? We have known that there exists the quasi-optimal  error estimate with  $\mathcal{O}\left(h^{4-\gamma}\right)$ convergence for \eqref{n1.2} by PQC in  \cite[p.\,525]{Aikinson:09} or \cite{Hoog:73}.
However, it is still not an optimal  error estimate  when the singular point coincides with an element junction point.
Developed  the techniques of  hypersingular integral \cite{GFTJZ:18,LiSun:10,WL:05}, we will provide  an optimal error  $\mathcal{O}\left(h^4\eta_i^{-\gamma}\right)$, $\eta_i=\min\left\{x_i-a,b-x_i\right\}$
for the integral \eqref{n1.2} with weakly singular kernels by PQC. Then the main purpose of the paper is the derivation of an optimal global convergence estimate with $\mathcal{O}\left(h^3\right)$ for nonlocal problems \eqref{n1.1} by PQC.

The paper is organized as follows. In the next section, we provide the discretization schemes for the integral \eqref{n1.2} and nonlocal model \eqref{n1.1}, respectively.
In Section 3, we study the local truncation error for integral \eqref{n1.2} by PLC and PQC. The  global convergence rate  for nonlocal model  \eqref{n1.1} by PLC and PQC, respectively, are detailed proved in Section 4. To show the effectiveness of the presented schemes, results of numerical experiments are reported in Section 5.
In particularity, some simulations for two-dimensional nonlocal problems with  nonsmooth kernels in  nonconvex   polygonal  domain are performed.
Finally, we conclude the paper with some remarks on the presented results.

\section{Collocation method and numerical schemes}\label{sec:1}
To elucidate the superconvergence phenomenon, we use  the piecewise linear and quadratic  polynomial collocation method  to approach the nonlocal model  \eqref{n1.1}.
Let us first  consider  the weakly singular integral  \eqref{n1.2}.
\subsection{Collocation method for integral \eqref{n1.2}}\label{subsection2.1}
In \cite{Aikinson:67} the author already provided  integral formulas to compute  the weakly singular integral \eqref{n1.2}
by the piecewise polynomial collocation.
Here, for the sake of theorems,  we   should explicitly express the coefficients of the  quadrature schemes  by integral formulas.

{\bf Case I: PLC for integral \eqref{n1.2}.}
Let  $a=x_{0}<x_{1}<x_2 \cdots <x_{N-1}<x_{N}=b$ be  a partition  with the uniform  mesh step $h=(b-a)/N$.
Let the piecewise linear basis  function $\phi_{j}(x)$ be defined by  \cite[p.\,484]{Aikinson:09}.
Then  the piecewise linear  interpolation $I_1(a,b,x)$ of ({\ref{n1.2}}) is
\begin{equation*}
\begin{split}
&I_1(a,b,x_{i})
=  \int^{b}_{a} \frac{ \sum^{N}_{j=0} u(x_j) \phi_{j}(y)}{|x_{i}-y|^{\gamma}} dy
= \sum^{N-1}_{j=0} \int^{x_{j+1}}_{x_{j}} \frac{ u(x_{j+1})\phi_{j+1}(y) +u(x_j)\phi_{j}(y) }{|x_{i}-y|^{\gamma}} dy \\
&=  \sum^{N-1}_{j=1}\! u(x_j)\! \int^{x_{j+1}}_{x_{j-1}} \!\!\! \frac{ \phi_{j}(y)}{|x_{i}-y|^{\gamma}} dy + u(x_0)\! \int^{x_{1}}_{x_{0}}
     \frac{ \phi_{0}(y)}{|x_{i}-y|^{\gamma}} dy
     + u(x_N) \int^{x_{N}}_{x_{N-1}} \frac{ \phi_{N}(y)}{|x_{i}-y|^{\gamma}} dy,
\end{split}
\end{equation*}
i.e.,
\begin{equation}\label{n2.1}
\begin{split}
I_1(a,b,x_{i})=\sigma_{h,\gamma}\left[\sum^{N-1}_{j=1}g_{|i-j|}u(x_j)  +\alpha_iu(x_0)+\alpha_{N-i}u(x_N) \right]
\end{split}
\end{equation}
with $\sigma_{h,\gamma}=\frac{h^{1-\gamma}}{(2-\gamma)(1-\gamma)}$.
Using  integral formulas of \cite{Aikinson:67}, we can   explicitly derive the  internal values  coefficients
$g_0=2$, $g_{k}=(k+1)^{2-\gamma}-2k^{2-\gamma}+(k-1)^{2-\gamma}, k\geq 1;$
and the boundary values  coefficients $\alpha_i=(i-1)^{2-\gamma}-i^{2-\gamma}+(2-\gamma)i^{1-\gamma}$, $i=1,2,\ldots  N-1.$


{\bf Case II: PQC for integral \eqref{n1.2}.}
Let  $a=x_{0}<x_{\frac{1}{2}}<x_{1}< \cdots <x_{\frac{2N-1}{2}}<x_{N}=b$ be  a  partition with the uniform mesh step $h=(b-a)/N$.
Let the piecewise quadratic basis function $\varphi_{j}(y)$ or $\varphi_{j+\frac{1}{2}}(y)$ be given in  \cite[p.\,499]{Aikinson:09}.
Let $u_Q(y)$ be the piecewise Lagrange quadratic interpolant of $u(y)$, i.e.,
\begin{equation}\label{n2.2}
  u_Q(y)=\sum^{N}_{j=0} u(x_j) \varphi_{j}(y)+\sum^{N-1}_{j=0} u(x_{j+\frac{1}{2}}) \varphi_{j+\frac{1}{2}}(y).
\end{equation}
Then we have  the following  piecewise quadratic  interpolation $I_2(a,b,x)$ of   ({\ref{n1.2}})
\begin{equation}\label{n2.3}
\begin{split}
I_2\left(a,b,x_{\frac{i}{2}}\right)
= & \int^{b}_{a} \frac{ u_Q(y)}{|x_{\frac{i}{2}}-y|^{\gamma}} dy\\
= & \sum^{N-1}_{j=1}\! u(x_j)\! \int^{x_{j+1}}_{x_{j-1}} \!\!\! \frac{ \varphi_{j}(y)}{|x_{\frac{i}{2}}-y|^{\gamma}} dy + u(x_0)\! \int^{x_{1}}_{x_{0}}
     \frac{ \varphi_{0}(y)}{|x_{\frac{i}{2}}-y|^{\gamma}} dy \\
     &+ u(x_N) \int^{x_{N}}_{x_{N-1}} \frac{ \varphi_{N}(y)}{|x_{\frac{i}{2}}-y|^{\gamma}} dy
     + \sum^{N-1}_{j=0}u(x_{j+\frac{1}{2}}) \int^{x_{j+1}}_{x_{j}} \frac{\varphi_{j+\frac{1}{2}}(y) }{|x_{\frac{i}{2}}-y|^{\gamma}} dy
\end{split}
\end{equation}
with   $1\leq i\leq 2N-1$.
We divide    \eqref{n2.3}  into two parts as follows
\begin{equation}\label{n2.4}
\begin{split}
I_2(a,b,x_i)
=&\eta_{h,\gamma}\left[\sum^{N-1}_{j=1}m_{|i-j|}u(x_j)  +   \sum^{N-1}_{j=0}q_{|i-j-\frac{1}{2}|-\frac{1}{2}}u(x_{j+\frac{1}{2}})   \right.\\
&\qquad\quad \left.+\beta_iu(x_0)+\beta_{N-i}u(x_N)\right],~i=1,2,\cdots,N-1;
\end{split}
\end{equation}
and
\begin{equation}\label{n2.5}
\begin{split}
I_2(a,b,x_{i+\frac{1}{2}})
=&\eta_{h,\gamma}\left[\sum^{N-1}_{j=1}p_{|i+\frac{1}{2}-j|-\frac{1}{2}}u(x_j)+\sum^{N-1}_{j=0}n_{|i-j|}u(x_{j+\frac{1}{2}})\right.\\
& \qquad \quad\left. +\gamma_iu(x_0)+\gamma_{N-i-1}u(x_N)  \right],~~i=0,1,\cdots,N-1
\end{split}
\end{equation}
with $\eta_{h,\gamma}=\frac{h^{1-\gamma}}{(3-\gamma)(2-\gamma)(1-\gamma)}$.

Here, from   integral formulas of \cite{Aikinson:67}, we can   explicitly compute $m_0=2(1+\gamma)$ and
$$m_k=4\left[ (k+1)^{3-\gamma}\!-\!(k-1)^{3-\gamma} \right] \!- \!(3-\gamma)\left[ (k+1)^{2-\gamma}+6k^{2-\gamma}+(k-1)^{2-\gamma}  \right],~ k\geq 1;$$
and $p_0=4\left[\left(\frac{3}{2}\right)^{3-\gamma}-\left(\frac{1}{2}\right)^{3-\gamma}\right]-(3-\gamma)\left[\left(\frac{3}{2}\right)^{2-\gamma}
+3\left(\frac{1}{2}\right)^{2-\gamma}\right]$,  $p_{k}=m_{k+\frac{1}{2}} $, $k\geq 1$.
Moreover, $q_{k}=-8\left((k+1)^{3-\gamma}-k^{3-\gamma}\right)+4(3-\gamma)\left((k+1)^{2-\gamma}+k^{2-\gamma}\right)$, $k \ge 0$;
and $n_{0}= (2-\gamma) 2^{\gamma+1}$, $n_{k}=q_{k-\frac{1}{2}} $, $k\geq 1.$
The boundary values coefficients
$$\beta_i
\!= 4\left[ i^{3-\gamma} \!- \!(i-1)^{3-\gamma}  \right]-(3-\gamma) \left[ 3i^{2-\gamma} + \left(i-1\right)^{2-\gamma}  \right]
 + (3-\gamma)(2-\gamma)i^{1-\gamma}, 1\!\leq i\leq\! N-1
$$
and $\gamma_{0}={(2-\gamma)(1-\gamma)} 2^{\gamma-1}$,  $\gamma_{i}=\beta_{i+\frac{1}{2}} $, $i\geq 1.$

\subsection{Collocation method for nonlocal model  \eqref{n1.1}}
Based on the discussion of the integral \eqref{n1.2},
we now provide the numerical schemes for nonlocal model  \eqref{n1.1}.

{\bf Case I:  PLC for  nonlocal model \eqref{n1.1}.} From  \eqref{n2.1},
E.q. (\ref{n1.1}) reduces to
\begin{equation}\label{n2.6}
\begin{split}
\int^{b}_{a} \frac{ u(x_i)}{\left| x_{i} - y \right|^{\gamma}} dy - I_1(a,b,x_{i}) = f(x_{i})+R_i, \quad i=1,2,\cdots,N-1,
\end{split}
\end{equation}
where the local truncation  error  $R_i=\mathcal{O}(h^2)$ will be proved in Lemma \ref{nlemma3.1}.
Let $u_i$ be the approximated value of $u(x_i)$ and  $f_i=f(x_i)$. Then the discretization scheme is
\begin{equation}\label{n2.7}
\begin{split}
&\sigma_{h,\gamma}\left[d_{i} u_{i}-\sum^{N-1}_{j=1}g_{|i-j|}u_j\right]=f_{i}+\sigma_h^1\left(\alpha_iu_0+\alpha_{N-i}u_N\right),\quad 1\leq i\leq N-1.
\end{split}
\end{equation}
Here the coefficients $\sigma_{h,\gamma}$, $\alpha_i$, $g_{|i-j|}$ are given in \eqref{n2.1}, and
$$d_{i}=(2-\gamma)\left[i^{1-\gamma}+(N-i)^{1-\gamma}\right].$$
For the convenience of implementation, we use the matrix form of the grid functions
$$U=(u_{1},u_{2},\cdots,u_{N-1})^{T},~~F=(f_{1},f_{2},\cdots,f_{N-1})^{T},$$
therefore, E.q. \eqref{n2.7}  can be rewritten as
\begin{equation}\label{n2.8}
  \sigma_{h,\gamma}(D-G)U=F+\sigma_{h,\gamma}H,
\end{equation}
where $D={\rm diag}\left(d_1,d_2,\ldots, d_{N-1}\right)$,
$G={\rm toeplitz}\left(g_0,g_1,\ldots, g_{N-2}\right)$ and
$$H=(\alpha_{1},\alpha_{2},\cdots,\alpha_{N-1})^{T}u_0+(\alpha_{N-1},\alpha_{N-2},\cdots,\alpha_{1})^{T}u_N.$$

{\bf Case II: PQC for  nonlocal model \eqref{n1.1}.}
From \eqref{n2.3}, we can rewrite (\ref{n1.1}) as
\begin{equation}\label{n2.9}
\begin{split}
\int^{b}_{a} \frac{ u\left(x_\frac{i}{2}\right)}{\left| x_{\frac{i}{2}} - y \right|^{\gamma}} dy - I_{2}\left(a,b,x_{\frac{i}{2}}\right)
= f\left(x_\frac{i}{2}\right)+R_\frac{i}{2}, ~~ 1\leq i\leq 2N-1.
\end{split}
\end{equation}
Here we will prove  that  the local truncation  error is  $R_\frac{i}{2}=\mathcal{O}\left(h^4\left(\eta_\frac{i}{2}\right)^{-\gamma}\right)$  in Theorem \ref{nlemma3.7}.
Let $u_\frac{i}{2}$ be the approximated value of $u(x_\frac{i}{2})$ and  $f_\frac{i}{2}=f(x_\frac{i}{2})$.
According to \eqref{n2.3}-\eqref{n2.5},  then the discretization scheme is the following systems
\begin{equation}\label{n2.10}
\begin{split}
&\eta_{h,\gamma}\left[d_{i} u_{i}-\sum^{N-1}_{j=1}m_{|i-j|}u_j-\sum^{N-1}_{j=0}q_{|i-j-\frac{1}{2}|-\frac{1}{2}}u_{j+\frac{1}{2}}\right]\\
&\quad=f_{i}+\eta_{h,\gamma}\left(\beta_iu_0+\beta_{N-i}u_N\right)\qquad\qquad{\rm for}~~1\leq i\leq N-1, \\\\
&\eta_{h,\gamma}\left[d_{i+\frac{1}{2}} u_{i+\frac{1}{2}}-\sum^{N-1}_{j=1}p_{|i+\frac{1}{2}-j|-\frac{1}{2}}u_j-\sum^{N-1}_{j=0}n_{|i-j|}u_{j+\frac{1}{2}}\right]\\
&\quad=f_{i+\frac{1}{2}}+\eta_{h,\gamma}\left(\gamma_iu_0+\gamma_{N-i-1}u_N\right)\qquad\,\,{\rm for}~~0\leq i\leq N-1,
\end{split}
\end{equation}
where
\begin{equation*}
\begin{split}
 d_{\frac{i}{2}} = (3-\gamma)(2-\gamma)\left( \left(\frac{i}{2}\right)^{1-\gamma} + \left(N-\frac{i}{2}\right)^{1-\gamma} \right), \quad i=1,2,\cdots,2N-1,
\end{split}
\end{equation*}
and  the coefficients $\eta_{h,\gamma}$, $\beta_i$, $\gamma_i$, $m_{|i-j|}$, $n_{|i-j|}$, $p_{|i+\frac{1}{2}-j|-\frac{1}{2}}$, $q_{|i-j-\frac{1}{2}|-\frac{1}{2}}$
are given  in \eqref{n2.4} and \eqref{n2.5}.
For the convenience of implementation, we use the matrix form of the grid functions $U=\left(u_{1},u_{2},\cdots,u_{N-1},u_{\frac{1}{2}},u_{\frac{3}{2}},\cdots,u_{N-\frac{1}{2}}\right)^{T}$
and similarly for $F$.
Therefore, we   can be rewrite  \eqref{n2.10} as the following systems
\begin{equation}\label{n2.11}
\begin{split}
\eta_{h,\gamma}\mathcal{A}U=F+\eta_{h,\gamma}K
\end{split}
\end{equation}
with
\begin{equation*}
\begin{split}
\mathcal{A}=\left [ \begin{matrix}
 \mathcal{D}_{1}  & 0\\
 0      &\mathcal{D}_{2}
 \end{matrix}
 \right ]
-
\left [ \begin{matrix}
  \mathcal{M}     & \mathcal{Q }  \\
 \mathcal{ P }        & \mathcal{N}
 \end{matrix}
 \right ]
\end{split}.
\end{equation*}
Here $\mathcal{D}_1={\rm diag}\left(d_1,d_2,\ldots, d_{N-1}\right)$, $\mathcal{D}_2={\rm diag}\left(d_\frac{1}{2},d_\frac{3}{2},\ldots, d_{N-\frac{1}{2}}\right)$,
$$\mathcal{M}={\rm toeplitz}\left(m_0,m_1,\ldots, m_{N-2}\right),~~\mathcal{N}={\rm toeplitz}\left(n_0,n_1,\ldots, n_{N-1}\right),$$ and
\begin{equation*}
\begin{split}
  K&=(\beta_{1},\beta_{2},\cdots,\beta_{N-1},\gamma_0,\gamma_1,\cdots,\gamma_{N-1})^{T}u_0\\
&\quad+(\beta_{N-1},\beta_{N-2},\cdots,\beta_{1},\gamma_{N-1},\gamma_{N-2},\cdots,\gamma_0)^{T}u_N.
\end{split}
\end{equation*}
The {\em rectangular matrices} $\mathcal{P}$, $\mathcal{Q}$ are defined by

\begin{equation*}
\begin{split}
\mathcal{P}=\left [ \begin{matrix}
p_{0}              & p_{1}              & p_{2}             &     \cdots & p_{N-3}   & p_{N-2}        \\
p_{0}              & p_{0}              & p_{1}             &     \cdots & p_{N-4}   & p_{N-3}         \\
p_{1}              & p_{0}              & p_{0}             &     \cdots & p_{N-5}   & p_{N-3}          \\
\vdots             & \vdots             &  \vdots           &     \ddots & \vdots    & \vdots            \\
p_{N-4}            & p_{N-5}            & p_{N-6}           &     \cdots & p_{0}     & p_{1}              \\
p_{N-3}            & p_{N-4}            & p_{N-5}           &     \cdots & p_{0}     & p_{0}               \\
p_{N-2}            & p_{N-3}            & p_{N-4}           &     \cdots & p_{1}     & p_{0}
 \end{matrix}
 \right ]_{N \times (N-1)}
\end{split}
\end{equation*}
and
\begin{equation*}
\begin{split}
\mathcal{Q}=\left [ \begin{matrix}
q_{0}              & q_{0}              & q_{1}             &     \cdots & q_{N-4}   & q_{N-3}   & q_{N-2}      \\
q_{1}              & q_{0}              & q_{0}             &     \cdots & q_{N-5}   & q_{N-4}   & q_{N-3}       \\
q_{2}              & q_{1}              & q_{0}             &     \cdots & q_{N-6}   & q_{N-5}   & q_{N-4}        \\
\vdots             & \vdots             &  \vdots           &     \ddots & \vdots    & \vdots    & \vdots           \\
q_{N-3}            & q_{N-4}            & q_{N-5}           &     \cdots & q_{0}     & q_{0}     &  q_{1}            \\
q_{N-2}            & q_{N-3}            & q_{N-4}           &     \cdots & q_{1}     & q_{0}     &  q_{0}
 \end{matrix}
 \right ]_{(N-1) \times N}
\end{split}.
\end{equation*}

\section{Local truncation error for integral \eqref{n1.2}}
As is well known,   an optimal  error estimate with  $\mathcal{O}\left(h^2\right)$ was proved by PLC and only $\mathcal{O}\left(h^3\right)$ convergence was established by PQC in \cite{Aikinson:67}.
To the best of our knowledge, the quasi-optimal  error estimate with  $\mathcal{O}\left(h^{4-\gamma}\right)$ convergence was provided by PQC, see \cite{Hoog:73} or \cite[p.\,525]{Aikinson:09}.
However, it is still not an optimal  error estimate  when the singular point coincides with an element junction point.
Based on the idea of \cite{GFTJZ:18,LiSun:10,WL:05}, we next provide  an optimal error  $\mathcal{O}\left(h^4\eta_i^{-\gamma}\right)$, $\eta_i=\min\left\{x_{i}-a,b-x_{i}\right\}$
for the integral \eqref{n1.2} by PQC.

Using Lagrange interpolation and the property of  weakly singular of integral \eqref{n1.2}, we obtain the following local truncation error for integral \eqref{n1.2} by PLC.
\begin{lemma}\cite{Aikinson:67}\label{nlemma3.1}
Let $I(a,b,x_{i})$ and $I_{1}(a,b,x_{i})$ be defined by (\ref{n1.2}) and (\ref{n2.1}),  respectively. If  $ u(x) \in C^{2}[a,b] $,  then
\begin{equation*}
\left|I(a,b,x_{i}) - I_{1}(a,b,x_{i}) \right| =\mathcal{O}(h^2).
\end{equation*}
\end{lemma}

\subsection{A few technical Lemmas}
Let us  first introduce some lemmas, which will be used to estimate the  local truncation error for integral \eqref{n1.2} by PQC.
\begin{lemma}\label{nlemma3.2}
Let $0<\gamma<1$, $ u(y) \in C^{4}[a,b] $ and $u_Q(y)$ be defined by \eqref{n2.2}.  Then
\begin{equation*}
 Q_{\frac{i}{2}}:=\int_{x_{\lceil\frac{i}{2}\rceil-1}}^{x_{\lfloor\frac{i}{2}\rfloor+1}} \frac{u(y)-u_Q(y)}
      {\left|x_{\frac{i}{2}}-y\right|^\gamma}dy
   =\mathcal{O}(h^{5-\gamma}),
\end{equation*}
where $i$ is a positive integer number, $\lfloor\frac{i}{2}\rfloor$ and $\lceil\frac{i}{2}\rceil$ denotes the greatest integer that is less than or equal to $\frac{i}{2}$ and the least integer that is greater than or equal to $\frac{i}{2}$, respectively.
\end{lemma}
\begin{proof}
If $i$ is even, we have
\begin{equation*}
\begin{split}
  \int_{x_{\frac{i}{2}}}^{x_{\frac{i}{2}+1}}\frac{\left(y-x_{\frac{i}{2}}\right)\left(y-x_{\frac{i+1}{2}}\right)
  \left(y-x_{\frac{i}{2}+1}\right)}{\left( y-x_{\frac{i}{2}} \right)^\gamma}dy
  =  & h^{4-\gamma}\int^{1}_{0}\frac{t\left(t-\frac{1}{2}\right)(t-1)}{t^{\gamma}} dt, \\
  \int^{x_{\frac{i}{2}}}_{x_{\frac{i}{2}-1}}\frac{\left(y-x_{\frac{i}{2}-1}\right)\left(y-x_{\frac{i-1}{2}}\right)
  \left(y-x_{\frac{i}{2}}\right)} {\left(x_{\frac{i}{2}}-y\right)^{\gamma}}dy
  =  & -h^{4-\gamma}\int^{1}_{0}\frac{t\left(t-\frac{1}{2}\right)(t-1)}{t^{\gamma}} dt.
\end{split}
\end{equation*}
From  Taylor expansion,  there exist $\xi_{\frac{i}{2}}\in [x_{\frac{i}{2}},x_{\frac{i}{2}+1}]$ and $\xi_{\frac{i}{2}-1}\in [x_{\frac{i}{2}-1},x_{\frac{i}{2}}]$ such that
$$u(y)-u_Q(y)=\frac{u^{(3)}\left(\xi_{\frac{i}{2}}\right) }{3!}\left(y-x_{\frac{i}{2}}\right)\left(y-x_{\frac{i+1}{2}}\right)\left(y-x_{\frac{i}{2}+1}\right)
~~\forall~y\in \left[x_{\frac{i}{2}},x_{\frac{i}{2}+1}\right] ; $$
and
$$u(y)-u_Q(y)=\frac{u^{(3)}\left(\xi_{\frac{i}{2}-1}\right) }{3!}\left(y-x_{\frac{i}{2}-1}\right)\left(y-x_{\frac{i-1}{2}}\right)\left(y-x_{\frac{i}{2}}\right)
~~\forall~y\in \left[x_{\frac{i}{2}-1},x_{\frac{i}{2}}\right]. $$
Then
\begin{equation*}
\begin{split}
 Q_{\frac{i}{2}}
&=\int_{x_{\frac{i}{2}-1}}^{x_{\frac{i}{2}+1}} \frac{u(y)-u_Q(y)}{\left|x_{\frac{i}{2}}-y\right|^\gamma}dy
  =\int_{x_{\frac{i}{2}-1}}^{x_{\frac{i}{2}}} \frac{u(y)-u_Q(y)}{\left(x_{\frac{i}{2}}-y\right)^\gamma}dy
    + \int_{x_{\frac{i}{2}}}^{x_{\frac{i}{2}+1}} \frac{u(y)-u_Q(y)}{\left(y-x_{\frac{i}{2}}\right)^\gamma}dy \\
& = \left(\frac{u^{(3)}\left(\xi_{\frac{i}{2}}\right) }{3!}-\frac{u^{(3)}\left(\xi_{\frac{i}{2}-1}\right) }{3!}\right) h^{4-\gamma}
     \int^{1}_{0}\frac{t(t-\frac{1}{2})(t-1)}{t^{\gamma}} dt \\
& = \frac{\gamma\left(u^{(3)}\left(\xi_{\frac{i}{2}}\right)-u^{(3)}\left(\xi_{\frac{i}{2}-1}\right)\right)}{12(4-\gamma)(3-\gamma)(2-\gamma)}h^{4-\gamma}
  =\mathcal{O}\left(h^{5-\gamma}\right).
\end{split}
\end{equation*}
If $i$ is odd, it yields
\begin{equation*}
\begin{split}
\int_{x_{\frac{i}{2}}}^{x_{\frac{i+1}{2}}}\frac{\left(y-x_{\frac{i-1}{2}}\right)\left(y-x_{\frac{i}{2}}\right)
  \left(y-x_{\frac{i+1}{2}}\right)}{\left( y-x_{\frac{i}{2}} \right)^\gamma}dy
& = h^{4-\gamma}\int^{\frac{1}{2}}_{0}\frac{\left(t+\frac{1}{2}\right)t\left(t-\frac{1}{2}\right)}{t^{\gamma}}dt, \\
\int^{x_{\frac{i}{2}}}_{x_{\frac{i-1}{2}}}\frac{\left(y-x_{\frac{i-1}{2}}\right)\left(y-x_{\frac{i}{2}}\right)
  \left(y-x_{\frac{i+1}{2}}\right)} {\left(x_{\frac{i}{2}}-y\right)^{\gamma}}dy
& =- h^{4-\gamma}\int^{\frac{1}{2}}_{0}\frac{\left(t+\frac{1}{2}\right)t\left(t-\frac{1}{2}\right)}{t^{\gamma}}dt.
\end{split}
\end{equation*}
Using  Taylor  expansion,  there exist $\xi\in \left[x_{\frac{i-1}{2}},x_{\frac{i+1}{2}}\right]$
$$u(y)-u_Q(y)=\frac{u^{(3)}\left(\xi\right) }{3!}\left(y-x_{\frac{i-1}{2}}\right)\left(y-x_{\frac{i}{2}}\right)\left(y-x_{\frac{i+1}{2}}\right)
~~\forall~y\in \left[x_{\frac{i-1}{2}},x_{\frac{i+1}{2}}\right]. $$
Therefore, we have
\begin{equation*}
\begin{split}
 Q_{\frac{i}{2}}
&=\int_{x_{\frac{i-1}{2}}}^{x_{\frac{i+1}{2}}} \frac{u(y)-u_Q(y)}{\left|x_{\frac{i}{2}}-y\right|^\gamma}dy
  =\int_{x_{\frac{i-1}{2}}}^{x_{\frac{i}{2}}} \frac{u(y)-u_Q(y)}{\left(x_{\frac{i}{2}}-y\right)^\gamma}dy
    + \int_{x_{\frac{i}{2}}}^{x_{\frac{i+1}{2}}} \frac{u(y)-u_Q(y)}{\left(y-x_{\frac{i}{2}}\right)^\gamma}dy \\
& =\frac{u^{(3)}\left(\xi\right) }{3!}h^{4-\gamma}
    \left(\int^{\frac{1}{2}}_{0}\frac{\left(t+\frac{1}{2}\right)t\left(t-\frac{1}{2}\right)}{t^{\gamma}}dt
    -\int^{\frac{1}{2}}_{0}\frac{\left(t+\frac{1}{2}\right)t\left(t-\frac{1}{2}\right)}{t^{\gamma}}dt\right)=0.
\end{split}
\end{equation*}
The proof is completed.
\end{proof}
\begin{lemma}\label{nlemma3.3}
Let $0<\gamma<1$, $ u(y) \in C^{4}[a,b] $ and $u_Q(y)$ be defined by \eqref{n2.2}.  Then
\begin{equation*}
\begin{split}
& Q_l:=\int^{x_{\lceil\frac{i}{2}\rceil-1}}_{x_{0}}\frac{u(y)-u_Q(y)}{\left|x_{\frac{i}{2}}-y\right|^\gamma}dy\\
&=-h^{4-\gamma}\cdot u^{(3)}\big(x_{\frac{i}{2}}\big) \sum^{\lceil\frac{i}{2}\rceil-1}_{m=1}\int^{1}_{0}\frac{t\left(t-\frac{1}{2}\right)(t-1)}{\left(\frac{i}{2}-m+t\right)^{\gamma}} dt
+\mathcal{O}(h^4)\left(x_{\frac{i}{2}}-a\right)^{1-\gamma}+\mathcal{O}(h^{5-\gamma}),
\end{split}
\end{equation*}
where $i$ is a positive integer number and  $\lceil\frac{i}{2}\rceil$ denotes the least integer that is greater than or equal to $\frac{i}{2}$.
\end{lemma}
\begin{proof}
Since $x_{\lceil\frac{i}{2}\rceil-1}=x_0$ with $i=1,2$, it yields  $Q_l=0$. Then we just need to estimate  $Q_l$ with $i\geq 3$.
For any $y\in[x_{m-1},x_m]$, using  Taylor expansion, there exist $\xi_{m}\in [x_{m-1},x_m]$ such that
$$u(y)-u_Q(y)=\frac{u^{(3)}(\xi_{m}) }{3!}(y-x_{m-1})\left(y-x_{m-\frac{1}{2}}\right)(y-x_{m}).  $$
For the sake of simplicity,  we take $w(\xi_{m})=\frac{u^{(3)}(\xi_{m}) }{3!}$ and
\begin{equation*}
\begin{split}
w(\xi_{m})
&=\left[w(\xi_{m})-w(x_{m})\right]+w(x_{m})\\
&=\left[w(\xi_{m})-w(x_{m})\right]+ w\left(x_{\frac{i}{2}}\right) +w'(\eta_{m})\left(x_{m}-x_{\frac{i}{2}}\right), \eta_{m}\in[x_m,x_{\frac{i}{2}}].
\end{split}
\end{equation*}
Then
\begin{equation*}
Q_l
=  \sum^{\lceil\frac{i}{2}\rceil-1}_{m=1}w(\xi_{m})
      \int^{x_{m}}_{x_{m-1}}\frac{ (y-x_{m-1})\left(y-x_{m-\frac{1}{2}}\right)(y-x_{m})}{\left(x_{\frac{i}{2}}-y\right)^{\gamma}}dy:=J_1+J_2+J_3
\end{equation*}
with
\begin{equation*}
\begin{split}
  J_1 =& \sum^{\lceil\frac{i}{2}\rceil-1}_{m=1}\left[w(\xi_{m})-w(x_{m})\right]
         \int^{x_{m}}_{x_{m-1}}\frac{ (y-x_{m-1})\left(y-x_{m-\frac{1}{2}}\right)(y-x_{m})}{\left(x_{\frac{i}{2}}-y\right)^{\gamma}}dy;\\
  J_2 =&  w\left(x_{\frac{i}{2}}\right)\sum^{\lceil\frac{i}{2}\rceil-1}_{m=1}\int^{x_{m}}_{x_{m-1}}
         \frac{ (y-x_{m-1})\left(y-x_{m-\frac{1}{2}}\right)(y-x_{m})}{\left(x_{\frac{i}{2}}-y\right)^{\gamma}}dy;\\
  J_3 =& \sum^{\lceil\frac{i}{2}\rceil-1}_{m=1} w'(\eta_{m})\left(x_{m}-x_{\frac{i}{2}}\right)\int^{x_{m}}_{x_{m-1}}
         \frac{ (y-x_{m-1})\left(y-x_{m-\frac{1}{2}}\right)(y-x_{m})}{\left(x_{\frac{i}{2}}-y\right)^{\gamma}}dy.
\end{split}
\end{equation*}
Using integration by parts and
 $\int^{1}_{0}\tau\left(\tau-\frac{1}{2}\right)\left(\tau-1\right)d\tau=0$, it  yields
\begin{equation}\label{n3.1}
\begin{split}
&  \int^{x_{m}}_{x_{m-1}}\!\!\!\frac{ (y-x_{m-1})\left(y-x_{m-\frac{1}{2}}\right)(y-x_{m})} {\left(x_{\frac{i}{2}}-y\right)^{\gamma}}dy
=  h^{4-\gamma}\!\!\int^{1}_{0}\!\!\frac{t\left(t-\frac{1}{2}\right)(t-1)}{\left(\frac{i}{2}-m+1-t\right)^{\gamma}} dt\\
&= -h^{4-\gamma}\int^{1}_{0}\frac{t\left(t-\frac{1}{2}\right)(t-1)}{\left(\frac{i}{2}-m+t\right)^{\gamma}} dt
= -\gamma h^{4-\gamma}\int^{1}_{0}\frac{\int^{t}_{0}\tau\left(\tau-\frac{1}{2}\right)(\tau-1)d\tau}{\left(\frac{i}{2}-m+t\right)^{1+\gamma}}dt.
\end{split}
\end{equation}
Moreover, we have
\begin{equation}\label{n3.2}
\begin{split}
&     \sum^{\lceil\frac{i}{2}\rceil-1}_{m=1}
      \int^{1}_{0}\left|\frac{\int^{t}_{0}\tau\left(\tau-\frac{1}{2}\right)(\tau-1)d\tau}{\left(\frac{i}{2}-m+t\right)^{1+\gamma}}\right| dt
\le   \sum^{\lceil\frac{i}{2}\rceil-1}_{m=1}\int^{1}_{0}\frac{1}{\left(\frac{i}{2}-m+t\right)^{1+\gamma}} dt \\
&\le  \sum^{\lceil\frac{i}{2}\rceil-1}_{m=1}\frac{1}{\left(\frac{i}{2}-m\right)^{1+\gamma}} = 2^{1+\gamma}\sum^{\lceil\frac{i}{2}\rceil-1}_{m=1}  \frac{1}{\left(i-2m\right)^{1+\gamma}}\\
&\leq  2^{1+\gamma}\sum^{i-1}_{m=2}  \frac{1}{\left(i-m\right)^{1+\gamma}}   =2^{1+\gamma}\sum^{i-2}_{m=1}  \frac{1}{m^{1+\gamma}}
\leq 2^{1+\gamma}\left( 1+ \frac{1}{\gamma}\right).
\end{split}
\end{equation}
Here, for the last inequality, we use
\begin{equation*}
\begin{split}
  \sum^{i-2}_{m=1}\frac{1}{m^{1+\gamma}}
   = & 1+\sum^{i-2}_{m=2}\frac{1}{m^{1+\gamma}}
   = 1+\sum^{i-2}_{m=2} \int^{m}_{m-1}\frac{1}{m^{1+\gamma}}dx
   \leq 1+\sum^{i-2}_{m=2} \int^{m}_{m-1}\frac{1}{x^{1+\gamma}}dx\\
   = & 1+ \int^{i-2}_{1}\frac{1}{x^{1+\gamma}}dx
   =   1+ \frac{1}{\gamma}-\frac{1}{\gamma\left(i-2\right)^{\gamma}}
   \leq 1+ \frac{1}{\gamma}.
\end{split}
\end{equation*}
From \eqref{n3.1} and \eqref{n3.2}, it leads to
\begin{equation*}
\begin{split}
&\left|J_1\right|\leq h^{5-\gamma}\max_{\eta\in[a,b]} |w'(\eta)|  \gamma 2^{1+\gamma}\left( 1+ \frac{1}{\gamma}\right) = \mathcal{O}(h^{5-\gamma});\\
&J_2 = -h^{4-\gamma}\cdot w\left(x_{\frac{i}{2}}\right) \sum^{\lceil\frac{i}{2}\rceil-1}_{m=1} \int^{1}_{0}\frac{t\left(t-\frac{1}{2}\right)(t-1)}{\left(\frac{i}{2}-m+t\right)^{\gamma}} dt.
\end{split}
\end{equation*}
Next we estimate the error term $J_3$.  Using  \eqref{n3.1} and \eqref{n3.2}, we have
\begin{equation*}
\begin{split}
|J_3|
&\leq \gamma h^{5-\gamma}\max_{\eta\in[a,b]} |w'(\eta)| \sum^{\lceil\frac{i}{2}\rceil-1}_{m=1}\left(\frac{i}{2}-m\right)\int^{1}_{0}\frac{1}{\left(\frac{i}{2}-m+t\right)^{1+\gamma}} dt\\
 &\le  \gamma h^{5-\gamma}\max_{\eta\in[a,b]} |w'(\eta)| \sum^{\lceil\frac{i}{2}\rceil-1}_{m=1}{\left(\frac{i}{2}-m\right)^{-\gamma}}.
\end{split}
\end{equation*}
We can check
\begin{equation*}
\begin{split}
&\sum^{\lceil\frac{i}{2}\rceil-1}_{m=1}{\left(\frac{i}{2}-m\right)^{-\gamma}}=\sum^{\frac{i}{2}-1}_{m=1}\frac{1}{m^\gamma}=\sum^{\frac{i}{2}-1}_{m=1}\int_{m-1}^m\frac{1}{m^\gamma}dx
\leq \sum^{\frac{i}{2}-1}_{m=1}\int_{m-1}^m\frac{1}{x^\gamma}dx\\
&= \int _0^{\frac{i}{2}-1}\frac{1}{x^\gamma}dx
 \leq \frac{1}{1-\gamma}\left(  \frac{i}{2}\right)^{1-\gamma}=  \frac{h^{\gamma-1}}{1-\gamma}\left(x_{\frac{i}{2}}-a\right)^{1-\gamma},~~i {\rm {~is ~even}}.
\end{split}
\end{equation*}
On the other hand, if $i$ is an odd, we have
\begin{equation*}
\begin{split}
&\sum^{\lceil\frac{i}{2}\rceil-1}_{m=1}{\left(\frac{i}{2}-m\right)^{-\gamma}}=\sum^{\frac{i-1}{2}}_{m=1}\frac{1}{\left(m-\frac{1}{2}\right)^\gamma}
=\left(\frac{1}{2}\right)^{-\gamma}+\sum^{\frac{i-1}{2}}_{m=2}\int_{m-\frac{3}{2}}^{m-\frac{1}{2}}\frac{1}{\left(m-\frac{1}{2}\right)^\gamma}dx\\
&\leq  \left(\frac{1}{2}\right)^{-\gamma}+\int_{\frac{1}{2}}^{\frac{i}{2}-1}\frac{1}{x^\gamma}dx
\leq  2\left(\frac{1}{2}\right)^{1-\gamma}+\frac{1}{1-\gamma}\left( \frac{i}{2} \right)^{1-\gamma}\leq \frac{3-2\gamma}{1-\gamma}\left( \frac{i}{2} \right)^{1-\gamma}\\
&= \frac{h^{\gamma-1}\left(3-2\gamma\right)}{1-\gamma}\left(x_{\frac{i}{2}}-a\right)^{1-\gamma}.
\end{split}
\end{equation*}
It implies  that
\begin{equation}\label{n3.3}
|J_3|\leq \!\gamma h^{5-\gamma}\max_{\eta\in[a,b]} |w'(\eta)|  \frac{h^{\gamma-1}\left(3-2\gamma\right)}{1-\gamma}\left(x_{\frac{i}{2}}-a\right)^{1-\gamma}
\!\!= \!\mathcal{O}\left(h^4\right)\left(x_{\frac{i}{2}}-a\right)^{1-\gamma}.
\end{equation}
The proof is completed.
\end{proof}
\begin{lemma}\label{nlemma3.4}
Let $0<\gamma<1$, $ u(y) \in C^{4}[a,b] $ and $u_Q(y)$ be defined by \eqref{n2.2}.  Then
\begin{equation*}
\begin{split}
 Q_r:&=\int^{x_{i}}_{x_{\lfloor\frac{i}{2}\rfloor+1}} \frac{u(y)-u_Q(y)}{\left|x_{\frac{i}{2}}-y\right|^\gamma}dy \\
&=h^{4-\gamma}\cdot w\left(x_{\frac{i}{2}}\right) \sum^{\lceil\frac{i}{2}\rceil-1}_{m=1} \int^{1}_{0}\frac{t\left(t-\frac{1}{2}\right)(t-1)}{\left(\frac{i}{2}-m+t\right)^{\gamma}} dt
+\mathcal{O}(h^4)\left(x_{\frac{i}{2}}-a\right)^{1-\gamma}\!\!+\mathcal{O}\left(h^{5-\gamma}\right),
\end{split}
\end{equation*}
where $i$ is a positive integer number and  $\lfloor\frac{i}{2}\rfloor$ denotes the greatest integer that is less than or equal to $\frac{i}{2}$.
\end{lemma}
\begin{proof}
Since $x_{\lfloor\frac{i}{2}\rfloor+1}=x_i$ with $i=1,2$, it yields  $Q_r=0$. Then we just need to estimate  $Q_r$ with $i\geq 3$.
For any $y\in[x_{m},x_{m+1}]$, using  Taylor  expansion, there exist $\xi_{m}\in [x_{m},x_{m+1}]$ such that
$$u(y)-u_Q(y)=\frac{u^{(3)}(\xi_{m}) }{3!}(y-x_{m})\left(y-x_{m+\frac{1}{2}}\right)(y-x_{m+1}).  $$
For the sake of simplicity,  we take $w(\xi_{m})=\frac{u^{(3)}(\xi_{m}) }{3!}$ and
\begin{equation*}
\begin{split}
w(\xi_{m})
&=\left[w(\xi_{m})-w(x_{m})\right]+w(x_{m})\\
&=\left[w(\xi_{m})-w(x_{m})\right]+ w\left(x_{\frac{i}{2}}\right) +w'(\eta_{m})\left(x_{m}-x_{\frac{i}{2}}\right), \eta_{m}\in[x_m,x_{\frac{i}{2}}].
\end{split}
\end{equation*}
Then
\begin{equation*}
Q_r
=  \sum_{m=\lfloor\frac{i}{2}\rfloor+1}^{i-1}w(\xi_{m})
      \int^{x_{m+1}}_{x_{m}}\frac{ (y-x_{m})\left(y-x_{m+\frac{1}{2}}\right)(y-x_{m+1})}{\left(y-x_{\frac{i}{2}}\right)^{\gamma}}dy:=\widetilde{J}_1+\widetilde{J}_2+\widetilde{J}_3
\end{equation*}
with
\begin{equation*}
\begin{split}
  \widetilde{J}_1 =& \sum_{m=\lfloor\frac{i}{2}\rfloor+1}^{i-1}\left[w(\xi_{m})-w(x_{m})\right]
         \int^{x_{m+1}}_{x_{m}}\frac{ (y-x_{m})\left(y-x_{m+\frac{1}{2}}\right)(y-x_{m+1})}{\left(y-x_{\frac{i}{2}}\right)^{\gamma}}dy;\\
  \widetilde{J}_2 =&  w\left(x_{\frac{i}{2}}\right)\sum_{m=\lfloor\frac{i}{2}\rfloor+1}^{i-1}
         \int^{x_{m+1}}_{x_{m}}\frac{ (y-x_{m})\left(y-x_{m+\frac{1}{2}}\right)(y-x_{m+1})}{\left(y-x_{\frac{i}{2}}\right)^{\gamma}}dy;\\
  \widetilde{J}_3 =& \sum_{m=\lfloor\frac{i}{2}\rfloor+1}^{i-1} w'(\eta_{m})\left(x_{m}-x_{\frac{i}{2}}\right)
         \int^{x_{m+1}}_{x_{m}}\frac{ (y-x_{m})\left(y-x_{m+\frac{1}{2}}\right)(y-x_{m+1})}{\left(y-x_{\frac{i}{2}}\right)^{\gamma}}dy.
\end{split}
\end{equation*}
Using integration by parts and
 $\int^{1}_{0}\tau\left(\tau-\frac{1}{2}\right)\left(\tau-1\right)d\tau=0$, it  yields
\begin{equation*}
\begin{split}
&  \int^{x_{m+1}}_{x_{m}}\frac{ (y-x_{m})\left(y-x_{m+\frac{1}{2}}\right)(y-x_{m+1})}{\left(y-x_{\frac{i}{2}}\right)^{\gamma}}dy\\
&=  h^{4-\gamma}\int^{1}_{0}\frac{t\left(t-\frac{1}{2}\right)(t-1)}{\left(m-\frac{i}{2}+t\right)^{\gamma}} dt
=\gamma h^{4-\gamma}\int^{1}_{0}\frac{\int^{t}_{0}\tau\left(\tau-\frac{1}{2}\right)(\tau-1)d\tau}{\left(m-\frac{i}{2}+t\right)^{1+\gamma}}dt.
\end{split}
\end{equation*}
Moreover, from  (\ref{n3.2}), we have
\begin{equation*}
\begin{split}
&     \sum_{m=\lfloor\frac{i}{2}\rfloor+1}^{i-1}
      \int^{1}_{0}\left|\frac{\int^{t}_{0}\tau\left(\tau-\frac{1}{2}\right)(\tau-1)d\tau}{\left(m-\frac{i}{2}+t\right)^{1+\gamma}}\right|dt
      \le   \sum_{m=\lfloor\frac{i}{2}\rfloor+1}^{i-1}\int^{1}_{0}\frac{1}{\left(m-\frac{i}{2}+t\right)^{1+\gamma}} dt \\
&=  \sum^{\lceil\frac{i}{2}\rceil-1}_{m=1}\int^{1}_{0}\frac{1}{\left(\frac{i}{2}-m+t\right)^{1+\gamma}} dt
\leq 2^{1+\gamma}\left( 1+ \frac{1}{\gamma}\right).
\end{split}
\end{equation*}
According to the above equations, there exists
\begin{equation*}
\begin{split}
\left|\widetilde{J}_1\right|&\leq h^{5-\gamma}\max_{\eta\in[a,b]} |w'(\eta)|  \gamma 2^{1+\gamma}\left( 1+ \frac{1}{\gamma}\right) = \mathcal{O}(h^{5-\gamma});\\
\widetilde{J}_2&=  h^{4-\gamma}w\left(x_{\frac{i}{2}}\right)\sum_{m=\lfloor\frac{i}{2}\rfloor+1}^{i-1}
   \int^{1}_{0}\frac{t\left(t-\frac{1}{2}\right)(t-1)}{\left(m-\frac{i}{2}+t\right)^{\gamma}} dt\\
&= h^{4-\gamma}\cdot w\left(x_{\frac{i}{2}}\right) \sum^{\lceil\frac{i}{2}\rceil-1}_{m=1} \int^{1}_{0}\frac{t\left(t-\frac{1}{2}\right)(t-1)}{\left(\frac{i}{2}-m+t\right)^{\gamma}} dt.
\end{split}
\end{equation*}
Next we estimate the error term $\widetilde{J}_3$.  From \eqref{n3.1} and \eqref{n3.2}, we have
\begin{equation*}
\begin{split}
|\widetilde{J}_3|
&\leq \gamma h^{5-\gamma}\max_{\eta\in[a,b]} |w'(\eta)| \sum_{m=\lfloor\frac{i}{2}\rfloor+1}^{i-1}\left(m-\frac{i}{2}\right)\int^{1}_{0}\frac{1}{\left(m-\frac{i}{2}+t\right)^{1+\gamma}} dt\\
&= \gamma h^{5-\gamma}\max_{\eta\in[a,b]} |w'(\eta)| \sum^{\lceil\frac{i}{2}\rceil-1}_{m=1}\left(\frac{i}{2}-m\right)\int^{1}_{0}\frac{1}{\left(\frac{i}{2}-m+t\right)^{1+\gamma}} dt\\
 &\le  \gamma h^{5-\gamma}\max_{\eta\in[a,b]} |w'(\eta)| \sum^{\lceil\frac{i}{2}\rceil-1}_{m=1}{\left(\frac{i}{2}-m\right)^{-\gamma}}.
\end{split}
\end{equation*}
The similar arguments can be performed as   \eqref{n3.3}, we get
\begin{equation*}
|\widetilde{J}_3|\leq \gamma h^{5-\gamma}\max_{\eta\in[a,b]} |w'(\eta)|  \frac{h^{\gamma-1}\left(3-2\gamma\right)}{1-\gamma}\left(x_{\frac{i}{2}}-a\right)^{1-\gamma} = \mathcal{O}(h^4)\left(x_{\frac{i}{2}}-a\right)^{1-\gamma}.
\end{equation*}
The proof is completed.
\end{proof}
\begin{lemma}\label{nlemma3.5}
Let $0<\gamma<1$, $ u(y) \in C^{4}[a,b] $ and $u_Q(y)$ be defined by \eqref{n2.2}.  Then
\begin{equation*}
\begin{split}
 Q_{c}:&=\int_{x_{i}}^{x_{N}} \frac{u(y)-u_Q(y)}{\left|x_{\frac{i}{2}}-y\right|^\gamma}dy
        =\mathcal{O}\left(h^4\left(x_{\frac{i}{2}}-a\right)^{-\gamma}\right),~~1\leq i \leq N-1.
\end{split}
\end{equation*}
\end{lemma}
\begin{proof}
For any $y\in[x_{m},x_{m+1}]$, using  Taylor  expansion, there exist $\xi_{m}\in [x_{m},x_{m+1}]$ such that
$$u(y)-u_Q(y)=\frac{u^{(3)}(\xi_{m}) }{3!}(y-x_{m})\left(y-x_{m+\frac{1}{2}}\right)(y-x_{m+1}).  $$
For the sake of simplicity,  we taking $w(\xi_{m})=\frac{u^{(3)}(\xi_{m}) }{3!}$. Using integration by parts and
 $\int^{1}_{0}\tau\left(\tau-\frac{1}{2}\right)\left(\tau-1\right)d\tau=0$, we have
\begin{equation*}
\begin{split}
&  \left|\int^{x_{m+1}}_{x_{m}}\frac{(y-x_{m})\left(y-x_{m+\frac{1}{2}}\right)\left(y-x_{m+1}\right)}
       {\left(y-x_{\frac{i}{2}}\right)^{\gamma}}dy\right|
=  h^{4-\gamma}\left|\int^{1}_{0}\frac{t\left(t-\frac{1}{2}\right)(t-1)}{\left(t+m-\frac{i}{2}\right)^{\gamma}} dt\right| \\
&= \gamma h^{4-\gamma}\left|\int^{1}_{0}\frac{ \int^{t}_{0}\tau\left(\tau-\frac{1}{2}\right)(\tau-1)d\tau}
            {\left(t+m-\frac{i}{2}\right)^{1+\gamma}} dt\right|
\leq \gamma h^{4-\gamma}\int^{1}_{0}\frac{ 1}{\left(t+m-\frac{i}{2}\right)^{1+\gamma}} dt.
\end{split}
\end{equation*}
Moreover,
\begin{equation*}
\begin{split}
&\sum^{N-1}_{m=i}\int^{1}_{0}\frac{1}{\left(t+m-\frac{i}{2}\right)^{1+\gamma}}dt
=\frac{1}{\gamma}\sum^{N-1}_{m=i}\left[  \left(m-\frac{i}{2}\right)^{-\gamma}-\left(1+m-\frac{i}{2}\right)^{-\gamma} \right]\\
&=\frac{1}{\gamma} \left[\left(\frac{i}{2}\right)^{-\gamma}-\left(N-\frac{i}{2}\right)^{-\gamma} \right]\leq \frac{1}{\gamma} \left(\frac{i}{2}\right)^{-\gamma}
=\frac{1}{\gamma}h^{\gamma}\left(x_{\frac{i}{2}}-a\right)^{-\gamma}.
\end{split}
\end{equation*}
According to the above equations, we have
\begin{equation*}
\begin{split}
\left| Q_c \right|
=& \left|\sum^{N-1}_{m=i}w(\xi_m)\int^{x_{m+1}}_{x_{m}}\frac{ (y-x_{m})\left(y-x_{m+\frac{1}{2}}\right)\left(y-x_{m+1}\right)}{\left(y-x_{\frac{i}{2}}\right)^{\gamma}}dy\right|\\
\leq&\gamma h^{4-\gamma}\max_{\eta\in[a,b]}\left|w(\xi)\right|\sum^{N-1}_{m=i}\int^{1}_{0}\frac{1}{\left(t+m-\frac{i}{2}\right)^{1+\gamma}}dt\\
\leq& h^{4-\gamma}\max_{\eta\in[a,b]} \left|w(\xi)\right|  h^{\gamma}\left(x_{\frac{i}{2}}-a\right)^{-\gamma}
=\mathcal{O}\left(h^4\left(x_{\frac{i}{2}}-a\right)^{-\gamma}\right).
\end{split}
\end{equation*}
The proof is completed.
\end{proof}

\begin{lemma}\label{nlemma3.6}
Let $0<\gamma<1$, $ u(y) \in C^{4}[a,b] $ and $u_Q(y)$ be defined by \eqref{n2.2}. Then
\begin{equation*}
\begin{split}
 \widetilde{Q}_{c}:&=\int_{x_{0}}^{x_{i-N}} \frac{u(y)-u_Q(y)}{\left|x_{\frac{i}{2}}-y\right|^\gamma}dy
        =\mathcal{O}\left(h^4\left(x_{b-\frac{i}{2}}\right)^{-\gamma}\right),~~N+1\leq  i \leq  2N-1.
\end{split}
\end{equation*}
\end{lemma}
\begin{proof}
For any $y\in[x_{m},x_{m+1}]$, using  Taylor  expansion, there exist $\xi_{m}\in [x_{m},x_{m+1}]$ such that
$$u(y)-u_Q(y)=\frac{u^{(3)}(\xi_{m}) }{3!}(y-x_{m})\left(y-x_{m+\frac{1}{2}}\right)(y-x_{m+1}).  $$
For the sake of simplicity,  we taking $w(\xi_{m})=\frac{u^{(3)}(\xi_{m}) }{3!}$. Using integration by parts and
 $\int^{1}_{0}\tau\left(\tau-\frac{1}{2}\right)\left(\tau-1\right)d\tau=0$, we have
\begin{equation*}
\begin{split}
&  \left|\int^{x_{m+1}}_{x_{m}}\frac{(y-x_{m})\left(y-x_{m+\frac{1}{2}}\right)\left(y-x_{m+1}\right)}
       {\left(x_{\frac{i}{2}}-y\right)^{\gamma}}dy\right|
=  h^{4-\gamma}\left|\int^{1}_{0}\frac{t\left(t-\frac{1}{2}\right)(t-1)}{\left(\frac{i}{2}-m-t\right)^{\gamma}} dt\right| \\
&= \gamma h^{4-\gamma}\left|\int^{1}_{0}\frac{ \int^{t}_{0}\tau\left(\tau-\frac{1}{2}\right)(\tau-1)d\tau}
            {\left(\frac{i}{2}-m-t\right)^{1+\gamma}} dt\right|
\leq \gamma h^{4-\gamma}\int^{1}_{0}\frac{ 1}{\left(\frac{i}{2}-m-t\right)^{1+\gamma}} dt.
\end{split}
\end{equation*}
Moreover,
\begin{equation*}
\begin{split}
&\sum^{i-N-1}_{m=0}\int^{1}_{0}\frac{1}{\left(\frac{i}{2}-m-t\right)^{1+\gamma}}dt
=\frac{1}{\gamma}\sum^{i-N-1}_{m=0}\left[  \left(\frac{i}{2}-m-1\right)^{-\gamma}-\left(\frac{i}{2}-m\right)^{-\gamma} \right]\\
&=\frac{1}{\gamma} \left[\left(N-\frac{i}{2}\right)^{-\gamma}-\left(\frac{i}{2}\right)^{-\gamma} \right]\leq \frac{1}{\gamma} \left(N-\frac{i}{2}\right)^{-\gamma}
=\frac{1}{\gamma}h^{\gamma}\left(b-x_{\frac{i}{2}}\right)^{-\gamma}.
\end{split}
\end{equation*}
According to the above equations, we have
\begin{equation*}
\begin{split}
\left| \widetilde{Q}_c \right|
=& \left|\sum^{i-N-1}_{m=0}w(\xi_m)\int^{x_{m+1}}_{x_{m}}\frac{ (y-x_{m})\left(y-x_{m+\frac{1}{2}}\right)\left(y-x_{m+1}\right)}{\left(x_{\frac{i}{2}}-y\right)^{\gamma}}dy\right|\\
\leq&\gamma h^{4-\gamma}\max_{\eta\in[a,b]}\left|w(\xi)\right|\sum^{i-N-1}_{m=0}\int^{1}_{0}\frac{1}{\left(\frac{i}{2}-m-t\right)^{1+\gamma}}dt\\
\leq& h^{4-\gamma}\max_{\eta\in[a,b]} \left|w(\xi)\right|  h^{\gamma}\left(b-x_{\frac{i}{2}}\right)^{-\gamma}
=\mathcal{O}\left(h^4\left(b-x_{\frac{i}{2}}\right)^{-\gamma}\right).
\end{split}
\end{equation*}
The proof is completed.
\end{proof}

\subsection{Local truncation error for integral \eqref{n1.2} with PQC}
According to the above results,  we obtain  the following.
\begin{theorem}\label{nlemma3.7}
Let $I(a,b,x_{\frac{i}{2}})$ and $I_{2}(a,b,x_{\frac{i}{2}})$ be defined by (\ref{n1.2}) and (\ref{n2.4}), respectively.  Let $0<\gamma<1$, $ u(y) \in C^{4}[a,b] $ and $u_Q(y)$ be defined by \eqref{n2.2}. Then
\begin{equation*}
\left|I(a,b,x_{\frac{i}{2}}) - I_{2}(a,b,x_{\frac{i}{2}}) \right|=\int^{b}_{a} {\frac{u(y)-u_Q(y)}{\left|x_{\frac{i}{2}}-y\right|^\gamma} }dy
=\mathcal{O}\left(h^4\left(\eta_\frac{i}{2}\right)^{-\gamma}\right)+\mathcal{O}(h^{5-\gamma})
\end{equation*}
with $\eta_\frac{i}{2}=\min\left\{x_{\frac{i}{2}}-a,b-x_{\frac{i}{2}}\right\}, ~i=1,2,\cdots,2N-1$.
\end{theorem}
\begin{proof}
If $x_{\frac{i}{2}}\leq \frac{b-a}{2}$, then
\begin{equation*}
\begin{split}
\int^b_a \frac{u(y)-u_Q(y)}{|x_{\frac{i}{2}}-y|^\gamma}dy=Q_l+Q_{\frac{i}{2}}+Q_r+Q_c
\end{split}
\end{equation*}
with
\begin{equation*}
\begin{split}
 Q_l:=&\int^{x_{\lceil\frac{i}{2}\rceil-1}}_{x_{0}}\frac{u(y)-u_Q(y)}{\left|x_{\frac{i}{2}}-y\right|^\gamma}dy,~~~~~~\qquad\,\,
Q_{\frac{i}{2}}:=\int_{x_{\lceil\frac{i}{2}\rceil-1}}^{x_{\lfloor\frac{i}{2}\rfloor+1}} \frac{u(y)-u_Q(y)} {\left|x_{\frac{i}{2}}-y\right|^\gamma}dy,\\
Q_r:=&\int^{x_{i}}_{x_{\lfloor\frac{i}{2}\rfloor+1}} \frac{u(y)-u_Q(y)}{\left|x_{\frac{i}{2}}-y\right|^\gamma}dy,~~~~~~\qquad \quad
 Q_c:=\int_{x_{i}}^{x_{N}} \frac{u(y)-u_Q(y)}{\left|x_{\frac{i}{2}}-y\right|^\gamma}dy.
\end{split}
\end{equation*}
According to Lemmas \ref{nlemma3.2}-\ref{nlemma3.5}, we obtain
\begin{equation*}
\left| R_{i} \right|=\int^{b}_{a} {\frac{u(y)-u_Q(y)}{\left|x_{\frac{i}{2}}-y\right|^\gamma} }dy
=\mathcal{O}\left(h^4\left(x_{\frac{i}{2}}-a\right)^{-\gamma}\right)+\mathcal{O}(h^{5-\gamma}).
\end{equation*}
If $x_{\frac{i}{2}}\geq \frac{b-a}{2}$,  then
\begin{equation*}
\begin{split}
\int^b_a \frac{u(y)-u_Q(y)}{|x_{\frac{i}{2}}-y|^\gamma}dy=\widetilde{Q}_c+\widetilde{Q}_l+\widetilde{Q}_{\frac{i}{2}}+\widetilde{Q}_r
\end{split}
\end{equation*}
with
\begin{equation*}
\begin{split}
 \widetilde{Q}_c:=&\int^{x_{i-N}}_{x_{0}}\frac{u(y)-u_Q(y)}{\left|x_{\frac{i}{2}}-y\right|^\gamma}dy,~~~~~~\qquad \,\,
\widetilde{Q}_l:= \int_{x_{i-N}}^{x_{\lceil\frac{i}{2}\rceil-1}} \frac{u(y)-u_Q(y)}{\left|x_{\frac{i}{2}}-y\right|^\gamma}dy,~~~~~~\\
\widetilde{Q}_{\frac{i}{2}}:=&\int_{x_{\lceil\frac{i}{2}\rceil-1}}^{x_{\lfloor\frac{i}{2}\rfloor+1}} \frac{u(y)-u_Q(y)} {\left|x_{\frac{i}{2}}-y\right|^\gamma}dy,~\quad~~~~\quad
 \widetilde{Q}_r:=\int_{x_{\lfloor\frac{i}{2}\rfloor+1}}^{x_{N}} \frac{u(y)-u_Q(y)}{\left|x_{\frac{i}{2}}-y\right|^\gamma}dy.
\end{split}
\end{equation*}
According to the  Lemma \ref{nlemma3.6} and  the similar arguments can be performed as Lemmas \ref{nlemma3.2}-\ref{nlemma3.4}, we have
\begin{equation*}
\left| R_{i} \right|=\int^{b}_{a} {\frac{u(y)-u_Q(y)}{\left|x_{\frac{i}{2}}-y\right|^\gamma} }dy
=\mathcal{O}\left(h^4\left(b-x_{\frac{i}{2}}\right)^{-\gamma}\right)+\mathcal{O}(h^{5-\gamma}).
\end{equation*}
The proof is completed.
\end{proof}
\begin{remark}\label{nremark3.1}
If $s$ is not  an element junction point, e.g., $s\in \left(x_{\frac{i}{2}},x_{\frac{i+1}{2}}\right)$,  the similar arguments can be performed as Theorem \ref{nlemma3.7} by PQC, we have
\begin{equation*}
\left|I(a,b,s) - I_{2}(a,b,s) \right|=\int^{b}_{a} {\frac{u(y)-u_Q(y)}{\left|s-y\right|^\gamma} }dy
=\mathcal{O}(h^{4-\gamma}),
\end{equation*}
which  coincides with \cite{Hoog:73} or \cite[p.\,525]{Aikinson:09}.
\end{remark}

\section{Global convergence rate  for nonlocal problems \eqref{n1.1}}
In \cite{TWW:13} remains to be proved the convergence error estimate by PLC. Inspired by this observations, we derive
an optimal global convergence estimate  for such nonlocal problems with
$\mathcal{O}\left(h\right)$ and  $\mathcal{O}\left(h^3\right)$ by PLC and PQC, respectively.
\subsection{Global convergence rate for model \eqref{n1.1} with PLC}
A symmetric positive definite matrix with positive entries on the diagonal and nonpositive off-diagonal entries is called an $M$-matrix.
Then we have the following.
\begin{lemma} \label{nlemma4.1}
Let matrix $A=D-G $ be defined by \eqref{n2.8}.
Then $A$ is an $M$-matrix.
\end{lemma}
\begin{proof}
Let $A=\left\{ a_{i,j} \right\}_{i,j=1}^{N-1}$ with $N\geq 2$.
From \eqref{n2.8} and \eqref{n2.1} and Taylor  expansion, we have
\begin{equation*}
a_{i,j}=\left\{\begin{split}
& d_{i}-g_0 >0, & i = j, \\
& -g_{|i-j|} <0,      & i \neq j.
\end{split}
\right.
\end{equation*}

We next prove the matrix $A$ is  strictly  diagonally dominant by rows.
Using
$1\equiv\sum_{j=0}^{N}\phi_j(x),$
it yields
\begin{equation}\label{4.01}
\begin{split}
&\int^{b}_{a} \frac{ 1}{\left| x_{i} - y \right|^{\gamma}} dy -
 \int^{b}_{a} \frac{ \sum^{N-1}_{j=1}  \phi_{j}(y)}{\left| x_{i} - y \right|^{\gamma}} dy
=  \int^{b}_{a} \frac{\phi_{0}(x) }{\left| x_{i} - y \right|^{\gamma}} dy +
    \int^{b}_{a} \frac{\phi_{N}(x) }{\left| x_i - y \right|^{\gamma}} dy\\
&=\sigma_{h,\gamma}\rho_i \geq \frac{(2-\gamma)(1-\gamma)}{2}\sigma_{h,\gamma}\left[\frac{1}{i^{\gamma}}  +\frac{1}{\left(N-i\right)^{\gamma}}\right]
=\frac{h^{1-\gamma}}{2}\left[\frac{1}{i^{\gamma}}  +\frac{1}{\left(N-i\right)^{\gamma}}\right]
\end{split}
\end{equation}
with
\begin{equation*}
\begin{split}
\rho_i=& \left[ (i-1)^{2-\gamma} - i^{2-\gamma} + (2-\gamma)i^{1-\gamma} \right. \\
 &\qquad +  \left.(N-i-1)^{2-\gamma} - (N-i)^{2-\gamma} + (2-\gamma)(N-i)^{1-\gamma}\right].
\end{split}
\end{equation*}
From 
\begin{equation*}
\begin{split}
&(i-1)^{2-\gamma} - i^{2-\gamma} + (2-\gamma)i^{1-\gamma}
  = i^{2-\gamma}\left[\left(1-\frac{1}{i}\right)^{2-\gamma}-1+(2-\gamma)\frac{1}{i}\right]\\
& = i^{2-\gamma}\left[\frac{(2-\gamma)(1-\gamma)}{2!}\frac{1}{i^{2}}+(2-\gamma)(1-\gamma)\sum_{n=1}^\infty\prod_{k=1}^n\frac{k+\gamma-1}{(n+2)!}\frac{1}{i^{n+2}}\right]\\
&\ge i^{2-\gamma}\left[\frac{(2-\gamma)(1-\gamma)}{2!}\frac{1}{i^{2}}\right]= \frac{(2-\gamma)(1-\gamma)}{2i^{\gamma}}>0;
\end{split}
\end{equation*}
and
\begin{equation*}
\begin{split}
(N-i-1)^{2-\gamma} - (N-i)^{2-\gamma} + (2-\gamma)(N-i)^{1-\gamma} \ge \frac{(2-\gamma)(1-\gamma)}{2\left(N-i\right)^{\gamma}}>0,
\end{split}
\end{equation*}
thus we have 
\begin{equation*}
\begin{split}
\sum^{N-1}_{j=1}a_{i,j} =\rho_i >0,~~i=1,2\ldots N-1.
\end{split}
\end{equation*}
From the Gerschgorin circle theorem \cite[p.\,388]{Horn:13}, the eigenvalues of  $A$ are in the disks centered at $a_{i,i}$ with radius
$r_i$, i.e.,  the eigenvalues $\lambda$ of the matrix  $A$ satisfy
$$  |\lambda -a_{i,i} | \leq r_i=\sum^{N-1}_{j=1,j\neq i}|a_{i,j}|, $$
which yields
\begin{equation}\label{n4.1}
\begin{split}
&  \lambda_{\min}(A) \geq \min\{a_{i,i}-r_i\}=\min\rho_i\\
&  = \min\left\{\frac{(2-\gamma)(1-\gamma)}{2i^{\gamma}}  + \frac{(2-\gamma)(1-\gamma)}{2\left(N-i\right)^{\gamma}}\right\},~~i=1,2\ldots N-1.
\end{split}
\end{equation}
The proof is completed.
\end{proof}
\begin{theorem}\label{ntheorem4.2}
Let $u_{i}$ be the approximate solution of $u(x_{i})$ computed by the discretization scheme (\ref{n2.8}). Let $\varepsilon_{i}=u(x_{i})-u_{i}$. Then
$$||u(x_{i})-u_{i}||_\infty  = \mathcal{O}(h).$$
\end{theorem}
\begin{proof}
Let $\varepsilon_{i}=u(x_{i})-u_{i}$ with $\varepsilon_{0}=\varepsilon_{N}=0$.
Subtracting (\ref{n2.8}) from (\ref{n2.6}),  we get
$$ \frac{h^{1-\gamma}}{(2-\gamma)(1-\gamma)} \cdot A\varepsilon=R$$
where
$\varepsilon=[\varepsilon_1,\varepsilon_2,\cdots,\varepsilon_{N-1}]^T$
and similarly for  $R$ with $R_{i}=\mathcal{O}(h^2)$ in Lemma \ref{nlemma3.1}.

Let $\left|\varepsilon_{i_0}\right|:=||\varepsilon||_{\infty}=\max_{1\leq i\leq N-1}|\varepsilon_i|$ and $A=\left\{ a_{i,j} \right\}_{i,j=1}^{N-1}$.
From Lemma \ref{nlemma4.1}, it yields $a_{i,i}>0$ and $a_{i,j}<0$, $i\neq j$ and
\begin{equation*}
\begin{split}
\left|R_{i_0}\right|
=&      \sigma_{h,\gamma} \left|a_{i_{0},i_{0}}\varepsilon_{i_0}+\sum^{N-1}_{j=1,j\neq i_{0}}a_{i_{0},j}\varepsilon_{j}\right|
\ge \sigma_{h,\gamma}\left[ a_{i_{0},i_{0}}\left|\varepsilon_{i_{0}}\right|-\sum^{N-1}_{j=1,j\neq i_{0}}\left|a_{i_{0},j}\right|
         \left|\varepsilon_{j}\right| \right]\\
\ge&  \sigma_{h,\gamma}\left[ a_{i_{0},i_{0}}\left|\varepsilon_{i_{0}}\right|-\sum^{N-1}_{j=1,j\neq i_{0}}\left|a_{i_{0},j}\right|
         \left|\varepsilon_{i_{0}}\right|\right]
=   \sigma_{h,\gamma} \left[a_{i_{0},i_{0}}-\!\!\!\sum^{N-1}_{j=1,j\neq i_{0}}\left|a_{i_{0},j}\right|\right]\left|\varepsilon_{i_{0}}\right|\\
= & \left[\int^{b}_{a} \frac{\phi_{0}(x) }{\left| x_{i_0} - y \right|^{\gamma}} dy +
    \int^{b}_{a} \frac{\phi_{N}(x) }{\left| x_{i_0} - y \right|^{\gamma}} dy\right]\left|\varepsilon_{i_0}\right|
    =S_{i_0}\left|\varepsilon_{i_0}\right|.
\end{split}
\end{equation*}
From \eqref{4.01} and Lemma \ref{nlemma3.1}, we have
$$S_{i_0}\geq \frac{h^{1-\gamma}}{2}\left[\frac{1}{i_0^{\gamma}}  +\frac{1}{\left(N-i_0\right)^{\gamma}}\right]\geq h^{1-\gamma}N^{-\gamma}=\frac{h}{\left(b-a\right)^\gamma}
~~{\rm and}~~R_{i_0}=\mathcal{O}\left(h^2\right).$$
Then
\begin{equation*}
\begin{split}
 ||\varepsilon||_{\infty}
 =\left|\varepsilon_{i_0}\right| \le \frac{\left|R_{i_0}\right|}{S_{i_0}}= \mathcal{O}\left(h\right).
\end{split}
\end{equation*}
The proof is completed.
\end{proof}

\begin{remark}\label{nremark4.1}
Fredholm integral equations of the second kind model ($*$)
holds $ ||u(x_{i})-u_{i}||_{\infty}= \mathcal{O}(h^2)$ by PLC, see \cite{Aikinson:67} and \cite[p.\,522]{Aikinson:09}.
If $\lambda=1$ and $\gamma=0$, the model ($*$) is equivalent to the   nonlocal model  \eqref{n1.1}  with $\gamma=0$, i.e.,
$
 \int^b_a u(x)-u(y)dy  =f(x).
$
From Theorem \ref{ntheorem4.2}, it leads to the  interesting results $ ||u(x_{i})-u_{i}||_{\infty}= \mathcal{O}(h).$

\end{remark}

\subsection{Global convergence rate  for model \eqref{n1.1} with PQC}
We next consider the properties of the stiffness matrix $\mathcal{A}$ in \eqref{n2.11}.
\begin{lemma}\label{nlemma4.3}
Let the matrices $\mathcal{M}$,  $\mathcal{N}$, $\mathcal{P}$, $\mathcal{Q}$ be defined by \eqref{n2.11}.
Then $ \mathcal{M}$, $ \mathcal{N}$,  $\mathcal{P}$,  $\mathcal{Q}$ are positive matrices.
\end{lemma}
\begin{proof}
Using  Taylor  expansion, we have
\begin{equation}\label{n4.2}
\begin{split}
  (1+z)^\alpha
 & =  1+\alpha z + \frac{\alpha(\alpha-1)}{2!}z^{2}+\frac{\alpha(\alpha-1)(\alpha-2)}{3!}z^{3}+\cdots\\
 & =  1+ \sum^{\infty}_{n=1} \prod^{n}_{ k=1} \frac{\alpha+1-k}{n!} z^{n},~~| z|\leq 1,~~\alpha>0.
\end{split}
\end{equation}
We first  estimate the elements of $\mathcal{M}$. From  \eqref{n2.4} and $\eqref{n4.2}$, it yields  $m_0>0$ and
\begin{equation*}
\begin{split}
  m_{i}
  =  & 4 i^{3-\gamma}\left[\left(1+\frac{1}{i}\right)^{3-\gamma}-\left(1-\frac{1}{i}\right)^{3-\gamma}\right]\\
     &\quad -(3-\gamma)i^{2-\gamma}\left[\left(1+\frac{1}{i}\right)^{2-\gamma}+6+\left(1-\frac{1}{i}\right)^{2-\gamma}\right]\\
  =  & 2 \sum^{\infty}_{n=1} \prod^{2n+1}_{k=1} (4-k-\gamma)\frac{3-2n}{(2n+1)!} i^{2-2n-\gamma}\\
  =  & 2(3-\gamma)(2-\gamma)(1-\gamma)i^{-\gamma}\left[ \frac{1}{6}
       - \sum^{\infty}_{n=1} \prod^{2n}_{k=1} (1-k-\gamma)\frac{2n-1}{(2n+3)!} i^{-2n} \right]\\
  \ge&  2(3-\gamma)(2-\gamma)(1-\gamma)i^{-\gamma} \left(\frac{1}{6} -\frac{7}{60} \right)> 0,\quad i\geq1,
\end{split}
\end{equation*}
since
\begin{equation*}
\begin{split}
&  \sum^{\infty}_{n=1} \prod^{2n}_{k=1} (1-k-\gamma)\frac{2n-1}{(2n+3)!} i^{-2n}\\
&  \quad\le  \sum^{\infty}_{n=1} \frac{ \left(2n\right)!  (2n-1)}{(2n+3)!}  = \frac{1}{60} +\sum^{\infty}_{n=2} \frac{  (2n-1)}{(2n+3)(2n+2)(2n+1)}\\
&  \quad\le   \frac{1}{60} + \sum^{\infty}_{n=2} \frac{ 1}{(2n+1)(2n+3)} \le \frac{1}{60}+\frac{1}{10}=\frac{7}{60}.
\end{split}
\end{equation*}
Now we   estimate the elements of $\mathcal{P}$.  From  \eqref{n2.5} and $\eqref{n4.2}$ and the above estimate of $m_i$, we have
\begin{equation*}
\begin{split}
p_{i}
= & 4z^{3-\gamma}\left[ \left( 1+\frac{1}{z} \right)^{3-\gamma} - \left( 1-\frac{1}{z} \right)^{3-\gamma} \right]\\
  &\quad  -(3-\gamma)z^{2-\gamma}\left[ \left( 1+\frac{1}{z} \right)^{2-\gamma}+6+\left( 1-\frac{1}{z} \right)^{2-\gamma} \right] \\
\ge & \frac{1}{10}(3-\gamma)(2-\gamma)(1-\gamma)z^{-\gamma} > 0~~\quad~~{\rm with}~~\quad~~z=i+\frac{1}{2},~~i\geq 1.
\end{split}
\end{equation*}
On the other hand, using   \eqref{n2.5} and $\eqref{n4.2}$, we obtain
\begin{equation*}
\begin{split}
  p_0
\geq &4\left[\left(\frac{3}{2}\right)^{3-\gamma}-\left(\frac{1}{2}\right)^{3-\gamma}\right]\\
&\quad-(3-\gamma)\left[\left(\frac{3}{2}\right)^{2-\gamma}
     +3\left(\frac{1}{2}\right)^{2-\gamma}\right]-(3-\gamma)(2-\gamma)\left(\frac{1}{2}\right)^{1-\gamma}  \\
 = &  (3-\gamma)(2-\gamma)(1-\gamma)z^{-\gamma}\left[\frac{1}{6}
       + \sum^{\infty}_{n=1}\prod^{n}_{k=1}(k-1+\gamma)\right]\frac{n^2+2n+1}{(n+3)!} z^{-n} \\
 \ge &  \frac{1}{6}(3-\gamma)(2-\gamma)(1-\gamma)z^{-\gamma}  > 0~~\quad~~{\rm with}~~\quad~~z=\frac{3}{2}.
\end{split}
\end{equation*}
We next  estimate the elements of $\mathcal{Q}$. From  \eqref{n2.4} and $\eqref{n4.2}$, we obtain
\begin{equation*}
\begin{split}
  q_{i}
  = & -8(i+1)^{3-\gamma}\left[1-\left(1-\frac{1}{i+1}\right)^{3-\gamma}\right]\\
    &  +4(3-\gamma)(i+1)^{2-\gamma}\left[1+\left(1-\frac{1}{i+1}\right)^{2-\gamma}\right] \\
  = & 4(3-\gamma)(2-\gamma)(1-\gamma)(i+1)^{-\gamma}
     \left[\frac{1}{6}+  \sum^{\infty}_{n=1}\prod^{n}_{k=1}(k-1+\gamma) \frac{n+1}{(n+3)!}  (i+1)^{-n}\right] \\
  > & \frac{2}{3}(3-\gamma)(2-\gamma)(1-\gamma)(i+1)^{-\gamma} > 0, \quad i\geq 0.
\end{split}
\end{equation*}
We last  estimate the elements of $\mathcal{N}$. From  \eqref{n2.5} and $\eqref{n4.2}$, it yields $n_0>0$  and
\begin{equation*}
\begin{split}
  n_{i}
  = & -8\left(z+1\right)^{3-\gamma}\left[1-\left(1-\frac{1}{z+1}\right)^{3-\gamma}\right]\\
    & +4(3-\gamma)\left(z+1\right)^{2-\gamma}\left[1+\left(1-\frac{1}{z+1}\right)^{2-\gamma}\right]\\
  \ge & \frac{2}{3}(3-\gamma)(2-\gamma)(1-\gamma)\left(z+1\right)^{-\gamma} > 0 ~~\quad~~{\rm with}~~\quad~~z=i-\frac{1}{2},~~i\geq 1.
\end{split}
\end{equation*}
The proof is completed.
\end{proof}

\begin{lemma}\label{nlemma4.4}
Let $0<\gamma<1$ and $1\le i \le 2N-1$. Then
$$\int^{b}_{a} \frac{\varphi_{0}(x) }{|x_{\frac{i}{2}}-y|^{\gamma}} dy\geq\frac{1}{6} (1-\gamma)h\left(x_{\frac{i}{2}} -a\right)^{-\gamma};$$
and
$$\int^{b}_{a} \frac{\varphi_{N}(x) }{|x_{\frac{i}{2}}-y|^{\gamma}} dy\geq\frac{1}{6} (1-\gamma)h\left(b-x_{\frac{i}{2}} \right)^{-\gamma}.$$
\end{lemma}
\begin{proof}
If $i=1$, then
\begin{equation*}
\begin{split}
\int^{b}_{a} \frac{\varphi_{0}(x) }{|x_{\frac{i}{2}}-y|^{\gamma}} dy
&= \int^{x_{\frac{1}{2}}}_{x_{0}} \frac{\frac{x_{1}-y}{h}\frac{x_{1}-2y}{h} }{(x_{\frac{1}{2}}-y)^{\gamma}} dy
+ \int^{x_{1}}_{x_{\frac{1}{2}}} \frac{\frac{x_{1}-y}{h}\frac{x_{1}-2y}{h} }{(y-x_{\frac{1}{2}})^{\gamma}} dy= h^{1-\gamma} \frac{4}{3-\gamma} \frac{1}{2^{3-\gamma}}\\
&=\frac{1}{2\left(3-\gamma\right)}h\left(x_{\frac{i}{2}} -a\right)^{-\gamma}\geq \frac{1}{6} (1-\gamma)h\left(x_{\frac{i}{2}} -a\right)^{-\gamma}.
\end{split}
\end{equation*}
Let $\eta_{h,\gamma}$ be given in  \eqref{n2.4} and $i\geq 2$.  Using Taylor  expansion \eqref{n4.2}, it yields
\begin{equation*}
\begin{split}
& \int^{b}_{a} \frac{\varphi_{0}(x) }{|x_{\frac{i}{2}}-y|^{\gamma}} dy
= \eta_{h,\gamma}\left[ 4\left( \left(\frac{i}{2}\right)^{3-\gamma} - \left(\frac{i}{2}-1\right)^{3-\gamma}  \right) \right. \\
& \qquad\left. -(3-\gamma) \left( 3\left(\frac{i}{2}\right)^{2-\gamma} + \left(\frac{i}{2}-1\right)^{2-\gamma}  \right) + (3-\gamma)(2-\gamma)\left(\frac{i}{2}\right)^{1-\gamma}\right]\\
&= \eta_{h,\gamma}(3-\gamma)(2-\gamma)(1-\gamma)\left(\frac{i}{2}\right)^{-\gamma}\!\left[ \frac{1}{6}-\gamma\sum^{\infty}_{n=5}\frac{\prod^{n-3}_{k=2}(k-1+\gamma)(n-4)}{n!}\left(\frac{i}{2}\right)^{3-n}\right]\\
&\ge \frac{1}{6}\eta_{h,\gamma} (3-\gamma)(2-\gamma)(1-\gamma)^2\left(\frac{i}{2}\right)^{-\gamma}=\frac{1}{6} (1-\gamma)h\left(x_{\frac{i}{2}} -a\right)^{-\gamma}.
\end{split}
\end{equation*}
Here, for the last inequality, we use
\begin{equation*}
\begin{split}
 &\sum^{\infty}_{n=5}\frac{\prod^{n-3}_{k=2}(k-1+\gamma)(n-4)}{n!}\left(\frac{i}{2}\right)^{3-n}
  < \sum^{\infty}_{n=5}\frac{(n-3)!(n-4)}{n!}\\
 &\quad=\sum^{\infty}_{n=5}\frac{1}{(n-1)(n-2)}-2\sum^{\infty}_{n=5}\left(\frac{1}{n} -\frac{2}{n-1} +\frac{1}{n-2}  \right)
= \frac{1}{6}.
\end{split}
\end{equation*}
On the other hand, there exists
\begin{equation*}
  \int^{b}_{a} \frac{\varphi_{N}(x) }{\left| x_{\frac{i}{2}} - y \right|^{\gamma}} dy=\int^{b}_{a} \frac{\varphi_{0}(x) }{\left| x_{N-\frac{i}{2}} - y \right|^{\gamma}} dy
  \ge \frac{1}{6} (1-\gamma)h\left(b-x_{\frac{i}{2}} \right)^{-\gamma}.
\end{equation*}
The proof is completed.
\end{proof}
\begin{lemma}\label{nlemma4.5}
Let
\begin{equation*}
\begin{split}
\mathcal{A}=\left [ \begin{matrix}
 \mathcal{D}_{1}  & 0\\
 0      &\mathcal{D}_{2}
 \end{matrix}
 \right ]
-
\left [ \begin{matrix}
  \mathcal{M}     & \mathcal{Q }  \\
 \mathcal{ P }        & \mathcal{N}
 \end{matrix}
 \right ]
\end{split}.
\end{equation*}
Then $\mathcal{A}$ is  strictly  diagonally dominant by rows.
\end{lemma}
\begin{proof}
From  Lemma  \ref{nlemma4.4}, we know that $ \mathcal{M}$, $ \mathcal{N}$,  $\mathcal{P}$,  $\mathcal{Q}$ are positive matrices.
From  Lemma \ref{nlemma4.4} and the property of the interpolation operator, i.e.,  
$$1\equiv\sum_{j=0}^{N}\varphi_j(x)+\sum_{j=0}^{N-1}\varphi_{j+\frac{1}{2}}(x),$$
it yields
\begin{equation*}
\begin{split}
S_\frac{i}{2}:=&\int^{b}_{a} \frac{ 1}{\left| x_{\frac{i}{2}} - y \right|^{\gamma}} dy -
 \int^{b}_{a} \frac{ \sum^{N-1}_{j=1}  \varphi_{j}(y)+\sum^{N-1}_{j=0}  \varphi_{j+\frac{1}{2}}(y)}{\left| x_{\frac{i}{2}} - y \right|^{\gamma}} dy\\
= & \int^{b}_{a} \frac{\varphi_{0}(x) }{\left| x_{\frac{i}{2}} - y \right|^{\gamma}} dy +
    \int^{b}_{a} \frac{\varphi_{N}(x) }{\left| x_{\frac{i}{2}} - y \right|^{\gamma}} dy>0,~~i=1,2,\cdots, 2N-1.
\end{split}
\end{equation*}
Using  \eqref{n2.3}, \eqref{n2.9} with $u(x)\equiv 1$,
we can rewrite the above equation as the following  
\begin{equation}\label{n4.3}
\begin{split}
\eta_{h,\gamma}\left\{\left [ \begin{matrix}
 \mathcal{D}_{1}  & 0\\
 0      &\mathcal{D}_{2}
 \end{matrix}
 \right ]
-
\left [ \begin{matrix}
  \mathcal{M}     & \mathcal{Q }  \\
 \mathcal{ P }        & \mathcal{N}
 \end{matrix}
 \right ]\right\}U=\eta_{h,\gamma}K
 =S,
\end{split}
\end{equation}
where $U=\left(1,1,\cdots,1\right)^{T}$ and
$
 S=\left(S_{1},S_{2},\cdots,S_{N-1},S_{\frac{1}{2}},S_{\frac{3}{2}},\cdots,S_{N-\frac{1}{2}}\right)^{T}.
$
The proof is completed.
\end{proof}
\begin{remark}\label{nremark4.2}
From Lemma \ref{nlemma4.5}, we know that the matrix $\mathcal{A}$ is nonsingular \cite[p.\,23]{Varga:00} and the linear system \eqref{n2.11} has a unique solution.
\end{remark}

\begin{theorem}\label{ntheorem4.6}
Let $u_{i}$ be the approximate solution of $u(x_{i})$ computed by the discretization scheme \eqref{n2.10}.  Then
$$ ||u(x_{i})-u_{i}||_{\infty}= \mathcal{O}(h^{3}).$$
\end{theorem}
\begin{proof}
 Let $\epsilon_{i}=u(x_{i})\!-\!u_{i}$ with $\epsilon_{0}=\epsilon_{N}=0$. Subtracting (\ref{n2.10}) from (\ref{n2.9}), we get
\begin{equation*}
  \eta_{h,\gamma} \cdot \mathcal{A}\epsilon=R
\end{equation*}
with $\epsilon=\left(\epsilon_{1},\epsilon_{2},\cdots,\epsilon_{N-1},\epsilon_{\frac{1}{2}},\epsilon_{\frac{3}{2}},\cdots,\epsilon_{N-\frac{1}{2}}\right)^{T}$
and similarly for  $R$.

Upon relabeling and reorienting the  vectors $\epsilon$ and $R$ as
\begin{equation*}
\begin{split}
\widetilde{\epsilon}&=\left(\epsilon_{\frac{1}{2}},\epsilon_{1},\epsilon_{\frac{3}{2}},\epsilon_{2},\cdots,\epsilon_{N-1},\epsilon_{N-\frac{1}{2}}\right)^{T},\\
\widetilde{R}&=\left(R_{\frac{1}{2}},R_{1},R_{\frac{3}{2}},R_{2},\cdots,R_{N-1},R_{N-\frac{1}{2}}\right)^{T},
\end{split}
\end{equation*}
then the above equation can be recast as
\begin{equation*}
  \eta_{h,\gamma} \cdot \mathcal{\widetilde{A}}\widetilde{\epsilon}=\widetilde{R}.
\end{equation*}

Let $\left|\epsilon_{\frac{i_{0}}{2}}\right|:=||\epsilon||_{\infty}=\max_{1\leq i\leq 2N-1}|\epsilon_\frac{i}{2}|$ and $\mathcal{\widetilde{A}}=\left\{ a_{i,j} \right\}_{i,j=1}^{2N-1}$.
From Lemma \ref{nlemma4.4} and \eqref{n4.3}, it yields $a_{i,i}>0$ and $a_{i,j}<0$, $i\neq j$ and
\begin{equation*}
\begin{split}
\left|R_{\frac{i_{0}}{2}}\right|
=&      \eta_{h,\gamma} \left|a_{i_{0},i_{0}}\epsilon_{\frac{i_{0}}{2}}+\sum^{2N-1}_{j=1,j\neq i_{0}}a_{i_{0},j}\epsilon_{\frac{j}{2}}\right|
\ge \eta_{h,\gamma}\left[ a_{i_{0},i_{0}}\left|\epsilon_{\frac{i_{0}}{2}}\right|-\sum^{2N-1}_{j=1,j\neq i_{0}}\left|a_{i_{0},j}\right|
         \left|\epsilon_{\frac{j}{2}}\right| \right]\\
\ge&  \eta_{h,\gamma}\left[ a_{i_{0},i_{0}}\left|\epsilon_{\frac{i_{0}}{2}}\right|-\!\!\!\sum^{2N-1}_{j=1,j\neq i_{0}}\left|a_{i_{0},j}\right|
         \left|\epsilon_{\frac{i_{0}}{2}}\right|\right]
=   \eta_{h,\gamma} \left[a_{i_{0},i_{0}}-\!\!\!\sum^{2N-1}_{j=1,j\neq i_{0}}\left|a_{i_{0},j}\right|\right]\left|\epsilon_{\frac{i_{0}}{2}}\right|\\
= & \left[\int^{b}_{a} \frac{\varphi_{0}(x) }{\left| x_{\frac{i_0}{2}} - y \right|^{\gamma}} dy +
    \int^{b}_{a} \frac{\varphi_{N}(x) }{\left| x_{\frac{i_0}{2}} - y \right|^{\gamma}} dy\right]\left|\epsilon_{\frac{i_{0}}{2}}\right|
    =S_{\frac{i_0}{2}}\left|\epsilon_{\frac{i_{0}}{2}}\right|.
\end{split}
\end{equation*}
According to  Lemma \ref{nlemma4.4} and Theorem \ref{nlemma3.7}, we have
$$S_{\frac{i_0}{2}}\geq \frac{1}{6} (1-\gamma)h\left[\left(x_{\frac{i_0}{2}} -a\right)^{-\gamma}+\left(b-x_{\frac{i_0}{2}} \right)^{-\gamma} \right],$$
and
$$R_{\frac{i}{2}}=\mathcal{O}\left(h^4\left(\eta_\frac{i}{2}\right)^{-\gamma}\right)=\max\left\{\left(x_{\frac{i}{2}}-a\right)^{-\gamma},\left(b-x_{\frac{i}{2}}\right)^{-\gamma}\right\}\mathcal{O}\left(h^4\right).$$
Then
\begin{equation*}
\begin{split}
 ||\epsilon||_{\infty}
 =\left|\epsilon_{\frac{i_{0}}{2}}\right| \le \frac{\left|R_{\frac{i_{0}}{2}}\right|}{S_{\frac{i_0}{2}}}= \mathcal{O}\left(h^3\right).
\end{split}
\end{equation*}
The proof is completed.
\end{proof}

\section{Numerical results}
In this section, we numerical verify the above theoretical results including convergence rates. In particularly, some simulations for two-dimensional nonlocal problems with  nonsmooth kernels in  nonconvex   polygonal  domain are performed.
\subsection{Numerical example for 1D}
In this subsection,   the $l_\infty$ norm is used to measure the numerical errors.
\begin{example}\label{ex:01}
To numerically confirm the result of Lemma \ref{nlemma3.1} and  Theorem \ref{nlemma3.7}, we consider the integral \eqref{n1.2}  with $a=0,b=1$.
Here the test function is $u(x)=e^x$ and define $f(x)$ accordingly.
\end{example}

\begin{table}[!th]
\renewcommand{\captionfont}{\footnotesize}
  \centering
  \small
  \begin{tabular}{cccccccccc}
  \hline\noalign{\smallskip}
 \multicolumn{1}{c}{$\gamma$} &\multicolumn{1}{c}{h}&\multicolumn{2}{c}{$x=h$}
      &\multicolumn{2}{c}{$x=1/3$}
      &\multicolumn{2}{c}{$x=1/2$}\\
\cline{3-4}\cline{5-6}\cline{7-8}\cline{9-10}\noalign{\smallskip}
&  &error  &order &error  &order &error  &order \\
  \hline
 0.3  &1/64 &4.7106e-05&~~~~   &5.9669e-05&~~~~   &6.0480e-05 &~~~~    \\
 0.3  &1/128 &1.1667e-05&2.0135 &1.4706e-05&2.0206 &1.5163e-05 &1.9959  \\
 0.3  &1/256 &2.8996e-06&2.0085 &3.6442e-06&2.0127 &3.7975e-06 &1.9974  \\
 0.3  &1/512&7.2218e-07&2.0054 &9.0607e-07&2.0079 &9.5039e-07 &1.9985  \\
  \hline
 0.7  &1/64 &9.3738e-05&~~~~   &3.2028e-04&~~~~   &1.5912e-04 &~~~~    \\
 0.7  &1/128 &2.3178e-05&2.0159 &7.2583e-05&2.1416 &4.0977e-05 &1.9572  \\
 0.7  &1/256 &5.7477e-06&2.0117 &1.6660e-05&2.1232 &1.0487e-05 &1.9662  \\
 0.7  &1/512&1.4280e-06&2.0090 &3.8597e-06&2.1098 &2.6712e-06 &1.9730  \\
  \hline
  \end{tabular}
 \caption{Example \ref{ex:01}: The errors of numerical scheme (\ref{n2.1}) with PLC.}
 \label{TT01}
\end{table}

\begin{table}[!th]
\renewcommand{\captionfont}{\footnotesize}
  \centering
  \small
  \begin{tabular}{cccccccccc}
  \hline\noalign{\smallskip}
 \multicolumn{1}{c}{$\gamma$} &\multicolumn{1}{c}{h}
      &\multicolumn{2}{c}{$x=h$}
      &\multicolumn{2}{c}{$x=1/3$}
      &\multicolumn{2}{c}{$x=1/2$}\\
\cline{3-4}\cline{5-6}\cline{7-8}\cline{9-10}\noalign{\smallskip}
&  &error  &order &error  &order &error  &order \\
  \hline
 0.3  &1/64  & 2.0549e-10&~~~~    &2.4848e-09&~~~~   &1.2613e-11&~~~~   \\
 0.3  &1/128 &1.6878e-11&3.6059~~~~&1.8583e-10&3.7410 &7.4474e-13&4.0820  \\
 0.3  &1/256 &1.3696e-12&3.6233~~~~&1.4627e-11&3.6672 &4.6185e-14&4.0112  \\
 0.3  &1/512 &1.1147e-13&3.6190~~~~&1.1098e-12&3.7202 &2.6645e-15&4.1154  \\
   \hline
 0.7  &1/64  &1.5922e-09&         ~~&1.6352e-07&~~~~   &3.3388e-10&~~~~    \\
 0.7  &1/128 &1.5680e-10&3.3440~~~~&1.6506e-08&3.3084 &2.0851e-11&4.0011  \\
 0.7  &1/256 &1.5652e-11&3.3245~~~~&1.6815e-09&3.2951 &1.3038e-12&3.9993  \\
 0.7  &1/512 &1.5730e-12&3.3148~~~~&1.7039e-10&3.3028 &8.3489e-14&3.9649  \\
  \hline
  \end{tabular}
 \caption{Example \ref{ex:01}: The errors of numerical scheme (\ref{n2.3}) with PQC.}
 \label{TT02}
\end{table}

Table \ref{TT01} shows that the convergence with the local truncation error $\mathcal{O}\left(h^2\right)$ for scheme \eqref{n2.1} by PLC, which   is in agreement  Lemma  \ref{nlemma3.1}.
Table \ref{TT02} shows that the convergence with the local truncation error $\mathcal{O}\left(h^4\left(\eta_i\right)^{-\gamma}\right)$, $\eta_i=\min\left\{x_{\frac{i}{2}}-a,b-x_{\frac{i}{2}}\right\}$  of  scheme \eqref{n2.3} by PQC.  It should be noted that if $x$  is  not  an element junction point (e.g., $x=\frac{1}{3}$), the errors reduce $\mathcal{O}(h^{4-\gamma})$, see Remark \ref{nremark3.1}.

\begin{example}\label{ex:02}
 Consider the  nonlocal problems \eqref{n1.1}
with a finite domain $a=0$, $b=1$. The exact solution is  $u(x)=e^x$ and the nonhomogeneous boundaries $u(0)=1$, $u(1)=e$.
Then define $f(x)$ accordingly.
\end{example}
\begin{table}[!th]
\renewcommand{\captionfont}{\footnotesize}
  \centering
  \small
  \begin{tabular}{cccccccccc}
  \hline\noalign{\smallskip}
 \multicolumn{1}{c}{h}&\multicolumn{2}{c}{$\gamma=0$} &\multicolumn{2}{c}{$\gamma=0.3$}
      &\multicolumn{2}{c}{$\gamma=0.7$}\\
\cline{2-3}\cline{4-5}\cline{6-7}\noalign{\smallskip}
  &error
  &order & error &order & error  &order\\
  \hline
  1/16  &8.9488e-03&~~~~   &1.0678e-02&~~~~   &1.4643e-02 &~~~~    \\
  1/32  &4.4746e-03&0.9999 &5.3149e-03&1.0065 &7.2511e-03 &1.0139 \\
  1/64  &2.2373e-03&1.0000 &2.6473e-03&1.0055 &3.5832e-03 &1.0170  \\
  1/128 &1.1187e-03&0.9999 &1.3201e-03&1.0039 &1.7732e-03 &1.0149  \\
  \hline
  \end{tabular}
 \caption{Example \ref{ex:02}: The errors of numerical scheme (\ref{n2.8}) with PLC.}
 \label{TT03}
\end{table}
%
\begin{table}[!th]
\renewcommand{\captionfont}{\footnotesize}
  \centering
  \small
  \begin{tabular}{cccccccccc}
  \hline\noalign{\smallskip}
 \multicolumn{1}{c}{h}&\multicolumn{2}{c}{$\gamma=0$} &\multicolumn{2}{c}{$\gamma=0.3$}
      &\multicolumn{2}{c}{$\gamma=0.7$}\\
\cline{2-3}\cline{4-5}\cline{6-7}\noalign{\smallskip}
  &error
  &order & error &order & error  &order\\
  \hline
  1/16  &4.3693e-07&~~~~   &2.5395e-07&~~~~   &4.6304e-07 &~~~~ \\
  1/32  &5.4621e-08&2.9999 &2.9901e-08&3.0863 &5.3886e-08 &3.1032 \\
  1/64  &6.8279e-09&2.9999 &3.5744e-09&3.0644 &6.2303e-09 &3.1125 \\
  1/128 &8.5191e-10&3.0027 &4.3270e-10&3.0463 &7.2423e-10 &3.1048 \\
  \hline
  \end{tabular}
 \caption{Example \ref{ex:02}: The errors of numerical scheme (\ref{n2.11}) with PQC.}
 \label{TT04}
\end{table}

Tables \ref{TT03} and  \ref{TT04} show that the linear and quadric  polynomial collocation method (\ref{n2.8}) and (\ref{n2.11}), respectively,
have first-order and third-order convergent, which are  in agreement  Theorems  \ref{ntheorem4.2}  and  \ref{ntheorem4.6}.

\subsection{Numerical example for 2D}
In this subsection,    the $l_\infty$ norm and the discrete $L^2$-norm, respectively, are used to measure the numerical errors.
\begin{example}\label{ex:06}
Let us  consider the following two-dimensional  nonlocal problems
\begin{equation*}
\int_{\Omega} \frac{u(x,y)-u(\bar{x},\bar{y})}{\left|\sqrt{(x-\bar{x})^2+(y-\bar{y}^2})\right|^{\gamma}} d\bar{x}d\bar{y}=f(x,y),
\end{equation*}
where the nonconvex polygonal domain is a  five-point star domain $\Omega$ in $(0, 2)\times(0, 2),$   and the exact solution is  $u(x, y)=e^{x^2}\cos(\pi y)$.
Then the nonhomogeneous boundaries condition and source function $f(x,y)$ are defined  accordingly.

In Fig. \ref{fig03},  the triangulations when $h=1/4$ and $h=1/8$ are depicted.

\end{example}
\begin{figure}[!th]
\renewcommand{\captionfont}{\footnotesize}
\centering
\subfigure[~]{
\begin{minipage}{6cm}
\centering
\includegraphics[width = 4.0cm,height =3.5cm]{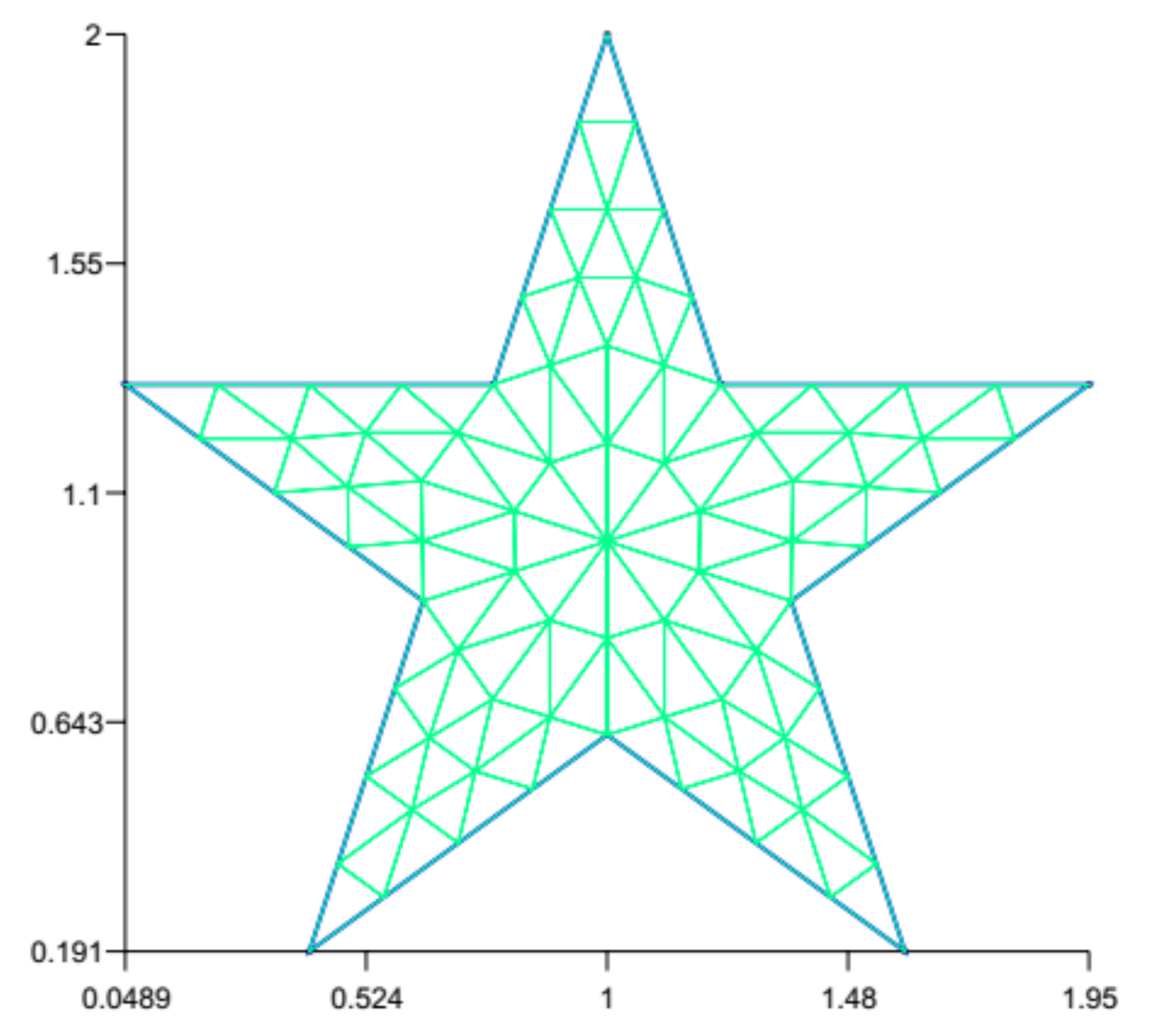}
\end{minipage}
}
\subfigure[~]{
\begin{minipage}{6cm}
\centering
\includegraphics[width = 4.0cm,height =3.5cm]{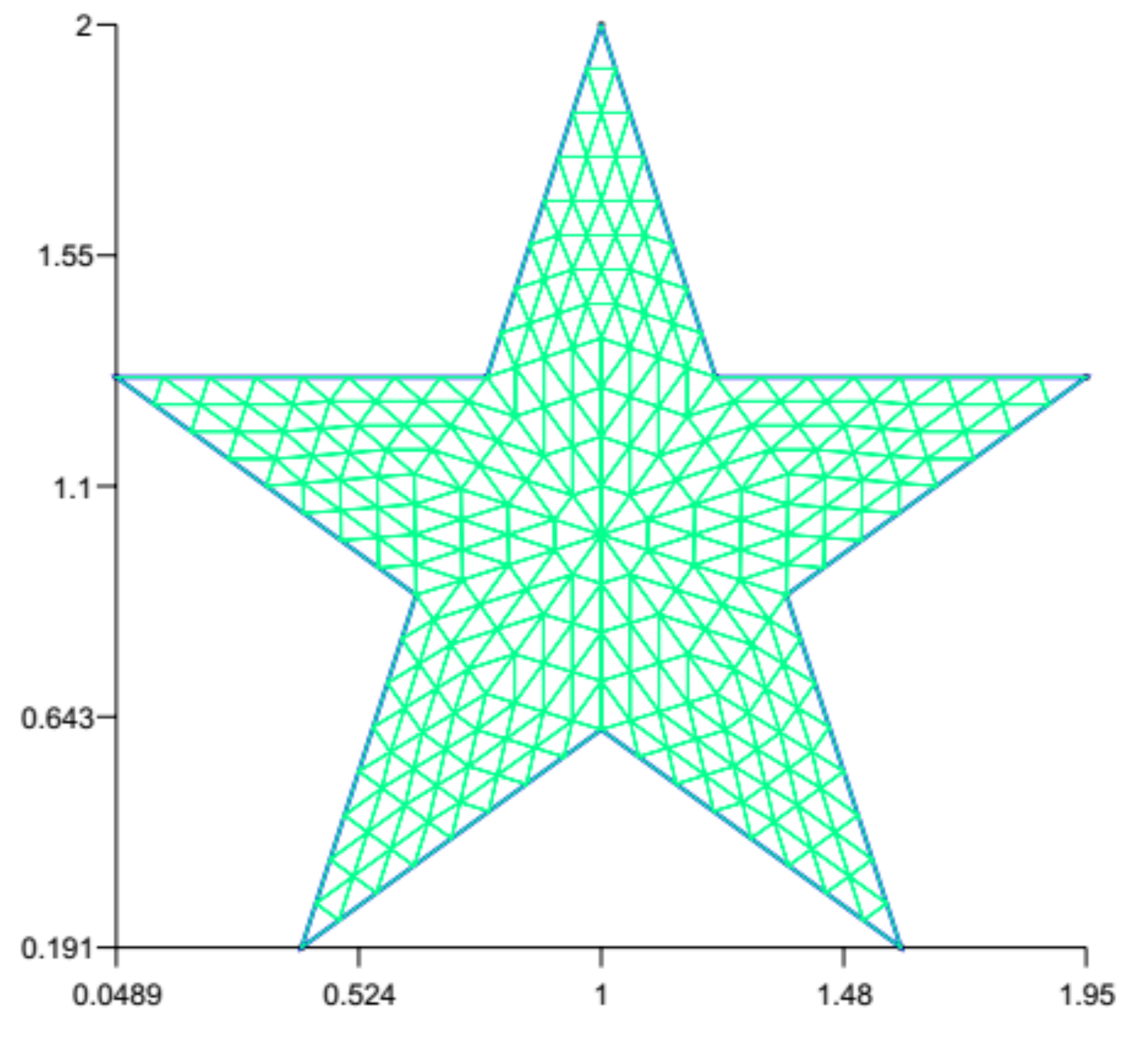}
\end{minipage}
}
\caption{The space meshes of Example \ref{ex:06}: (a) h=1/4, (b) h=1/8 }
\label{fig03}
\end{figure}

\begin{table}[!th]
\renewcommand{\captionfont}{\footnotesize}
  \centering
  \small
  \begin{tabular}{cccccccccc}
  \hline\noalign{\smallskip}
 \multicolumn{1}{c}{h}&\multicolumn{4}{c}{$\gamma=0.3$}
      &\multicolumn{4}{c}{$\gamma=0.7$}\\
\cline{2-5}\cline{6-9}\noalign{\smallskip}
  &$||\cdot||_{L^2}$  &order &$||\cdot||_{l_{\infty}}$  &order &$||\cdot||_{L^2}$ &order &$||\cdot||_{l_{\infty}}$  &order\\
  \hline
  1/4  &8.145e-04&~~~~   &7.920e-03&~~~~   &2.334e-03 &~~~~   &2.429e-02 &~~~~ \\
  1/8  &3.617e-04&1.17 &2.720e-03&1.54 &1.035e-03 &1.17 &8.728e-03 &1.47 \\
  1/16 &1.649e-04&1.13 &1.024e-03&1.40 &4.589e-04 &1.17 &3.158e-03 &1.46 \\
  1/32 &7.915e-05&1.05 &4.217e-04&1.27 &2.159e-04 &1.08 &1.244e-03 &1.34 \\
  \hline
  \end{tabular}
 \caption{Example \ref{ex:06}:  The errors of numerical simulations by PLC.}
 \label{TT13}
\end{table}

\begin{table}[!th]
\renewcommand{\captionfont}{\footnotesize}
  \centering
  \small
  \begin{tabular}{cccccccccc}
  \hline\noalign{\smallskip}
 \multicolumn{1}{c}{h}&\multicolumn{4}{c}{$\gamma=0.3$}
      &\multicolumn{4}{c}{$\gamma=0.7$}\\
\cline{2-5}\cline{6-9}\noalign{\smallskip}
  &$||\cdot||_{L^2}$  &order &$||\cdot||_{l_{\infty}}$  &order &$||\cdot||_{L^2}$ &order &$||\cdot||_{l_{\infty}}$  &order\\
  \hline
  1/4  &1.243e-05&~~~~   &9.466e-05&~~~~   &3.205e-05 &~~~~   &3.441e-04 &~~~~ \\
  1/8  &1.029e-06&3.59 &7.083e-06&3.74 &3.585e-06 &3.16 &3.184e-05 &3.43 \\
  1/16 &1.084e-07&3.24 &6.368e-07&3.47 &4.274e-07 &3.06 &3.031e-06 &3.39 \\
  1/32 &1.281e-08&3.08 &6.601e-08&3.27 &5.237e-08 &3.02 &3.084e-07 &3.29 \\
  \hline
  \end{tabular}
 \caption{Example \ref{ex:06}: The errors of numerical simulations by PQC.}
 \label{TT14}
\end{table}
Table \ref{TT13} and Table \ref{TT14} show that the orders of accuracy are
$O(h)$ and $O(h^3)$ by PLC and PQC, respectively, in a a nonconvex polygonal domain.
Here $||\cdot||_{l_{\infty}}$ denotes the $l_\infty$ norm and $||\cdot||_{L^2}$ denotes  the discrete $L^2$-norm.
\section{Conclusion}
In this work, we first derive an optimal error estimate  for  weakly singular integral \eqref{n1.2} by PQC  when the singular point coincides with an element junction point.
Then the sharp error estimate of piecewise linear and quadratic polynomial collocation  for nonlocal problems \eqref{n1.1}  are provided.
Hopefully,  an optimal error estimate of the $k$th-order Newton-Cotes rule   $O(h^k)$ for odd $k$ and $O(h^{k+1})$ for even $k$
can be obtained of nonlocal model \eqref{n1.1} by following the idea given in this paper.
Moreover, it is also provided a few technical analysis for two-dimensional nonlocal problems with singular kernels or  other  nonsmooth kernels.

\bibliographystyle{amsplain}

\end{document}